\documentclass[10pt,a4,onesided]{amsart}

 \usepackage{epsfig}
\usepackage{amssymb}
\usepackage{color}
\usepackage{mathrsfs}
\usepackage{hyperref}

\newcommand\N{{\mathbb N}}
\newcommand\R{{\mathbb R}}

\newcommand\Pp{{\mathbb P}}
\newcommand\Ee{{\mathbb E}}

\def\EE{{\mathcal E}}

\def\LL{{\mathcal L}}
\def\MM{{\mathcal M}}
\def\NN{{\mathcal N}}

\def\PP{{\mathcal P}}

\def\SS{{\mathcal S}}

\def\WW{{\mathcal W}}

\def\TTT{{\mathscr T}}

\def\eps{{\varepsilon}}

\newtheorem{theo}{Theorem}
\newtheorem{prop}[theo]{Proposition}
\newtheorem{lem}[theo]{Lemma}

\newtheorem{rem}[theo]{Remark}

\newcommand{\beqn}{\begin{equation}}
\newcommand{\eeqn}{\end{equation}}
\newcommand{\bear}{\begin{eqnarray}}
\newcommand{\eear}{\end{eqnarray}}
\newcommand{\bean}{\begin{eqnarray*}}
\newcommand{\eean}{\end{eqnarray*}}

\newcommand{\loi}{{\mathcal L}}

\newcommand{\indiq}{{\bf 1}}

\def\eps{{\varepsilon}}

\def\signcq{\bigskip \begin{center} {\sc Crist\'obal Qui\~ninao\par\vspace{3mm}
Universit\'e Pierre et Marie Curie \par
Laboratoire Jacques-Louis Lions, CNRS UMR 7598\par
4 place de Jussieu
F-75005, Paris\par
FRANCE\par
and Mathematical Neuroscience Team, CIRB\par 
College de France\par\vspace{3mm}
e-mail:} \tt{cristobal.quininao@college-de-france.fr}  \end{center}}

\begin{document}
%

\title[On a Poisson coupling for an Age-structured neuronal network]{A microscopic spiking neuronal network for the age-structured model}

\author{Cristobal Qui\~ninao}

\maketitle

\begin{center} {\bf Work in progress - Preliminary version of \today}
\end{center}

\begin{abstract}
We introduce a microscopic spiking network consistent with the age-structured/renewal equation proposed by Pakdaman, Perthame and Salort. It is a jump process interacting through a global activity variable with random delays. We show the well-posedness of the particle system and the mean-field equation. Moreover we show the propagation of chaos property and we quantify the rate of convergence under the assumption of exponential moments on the initial data.
\end{abstract}

\tableofcontents


\section{Introduction and main results}
\label{sec:intro}
\setcounter{equation}{0}
\setcounter{theo}{0}


In a series of remarkable papers, Pakdaman, Perthame and Salort (PPS)~\cite{MR2576373,MR3071416,MR3246939} introduced a versatile model for the large-scale dynamics of neuronal networks. These equations describe the probability distribution of the time elapsed since the last spike fired as an age-structured nonlinear PDE. Inspired by the dynamics of these macroscopic equations, we propose here a microscopic model describing the dynamics of a finite number of neurons, and that provides a realistic neural network model consistent with the PPS model, in the sense that in the thermodynamic limit, propagation of chaos and convergence to the PPS equation is proved.
\smallskip

In our neuronal network model, the state of each neuron $i$ is described by a $\R_+$-valued variable $X^{i,N}_t$ corresponding to the time elapsed since last discharge. Of course, this approach is quite different from classical literature, where the key variable is the voltage: this is an important originality of the PPS model. Neurons interact through the emission and reception of action potentials (or spikes), which are fast stereotyped trans-membrane current. In the PPS model, the spiking rate essentially depends on the global activity $M$ of the network. Specifically, a neuron with age $x$ (duration since it fired its last spike) fires a spike with an instantaneous intensity $a(x,M)$ where $M$ is the global activity of the network. Subsequently to the spike emission, two things happen: the age of the spiking neuron is reset to 0, and the global variable $M$ increases its value by an extra value of $\eps/N$. The coefficient $\eps$ represents the mean connectivity of the network. When no spikes occur, the global variable $M$ is supposed to decay exponentially to 0 at a constant rate $\alpha$.
\smallskip

For each $N\in\N$, let us consider a family $(\NN^1_t,\ldots,\NN^N_t)_{t\geq0}$ of i.i.d. standard Poisson processes. Let us also consider a family $(\tau_1,\ldots,\tau_N)$ of i.i.d. real valued random variables with probability law $b$. These coefficients represent delays in the transmission of information from the cell to whole network. Furthermore, we assume that the family of delays is independent of the Poisson processes. Typical examples are 
$$
b =  \, \delta_{\tau}(ds),\quad\tau\geq0 \qquad\hbox{or}\qquad
b =  \, c\,e^{-cs}\, ds,\quad c>0.
$$
A special case is when $\tau=0$, in that context the network is said to be in a no delay regime.
\smallskip

Throughout the paper we assume chaotic initial conditions, in the sense that the initial state of the neurons are independent and identically distributed random variables. Therefore, for $g_0$ and $m_0$ two independent probability measures on $\R_+$, $(g_0,m_0)$-\textit{chaotic initial states} consists in setting i.i.d. initial conditions for all neurons with common law equal to $g_0$, and setting independently, for the global activity, another random variable distributed as $m_0$.
\smallskip

Our aim is to understand the convergence of the $\R_+$-valued Markov processes $$(X^N_t)_{t\geq0}=(X^{1,N},\ldots, X^{N,N}_t)_{t\geq0},$$ solving, for each $i=1,\ldots,N$ and any $t\geq0$:
\begin{equation}\label{eq:syspar-X}
  X^{i,N}_t = X_0^{i,N}+t- \int_0^tX^{i,N}_{s-}\int_{0}^\infty \indiq_{\{u\leq a(X^{i,N}_{s-},M_{s-}^N)\}}\,\NN^i(du,ds),
\end{equation}
with the coupling given by the global variable
 \begin{equation}\label{eq:syspar-M}
  M_t^N = M_0^N-\alpha\Big[\int_0^t M_s^N\,ds -\frac{\eps}{N}\sum_{j=1}^N \int_{0}^{t}\int_{0}^\infty  \indiq_{\{u\leq a(X^{j,N}_{s_--\tau_j},M_{s_--\tau_j}^N)\}}\,\tilde\NN^j(du,ds)\Big],
\end{equation}
where $\tilde\NN^j_t$ is the shifted (in time) process $\NN^j_{t-\tau_j}$ extended by 0 for negative values of the time. These processes are a consistency restriction on the spiking times: \textit{when a neuron $j$ sends a signal at a time $t\geq0$, the global variable receives it only at instant $t+\tau_j$}. 

We emphasise that it is necessary to impose that the delays are compactly supported, i.e. on an interval of the type $[-\tau,0]$ for some fixed $\tau$. Otherwise the processes $X^{i,N}_t$ would depend on their whole past trajectory. Equations~\eqref{eq:syspar-X}-\eqref{eq:syspar-M} are stochastic differential equations on the infinite-dimensional space of continuous functions from $[-\tau,0]$ to $\R_+$, i.e. on the variables $\tilde X_t=(X_s,s\in [t-\tau,t])$ (see e.g. Da Prato-Zabczyk~\cite{da2014stochastic}).

Finally, we make the following physically reasonable assumption on the intensity spike function of the system: 
\begin{equation}\label{eq:syspar-H1}
\begin{cases}
 a(\cdot,\cdot)\text{ is a continuous non decreasing function in both variables,}
\\
 a(0,\cdot)\,=\,0,\quad a(\cdot,0)>0
\\
 a(x,m)\xrightarrow{x\rightarrow\infty} \infty,\quad\forall\,m\in \overline\R_+,
\end{cases}
\end{equation}
and impose a second consistency restriction
\begin{equation}\label{eq:syspar-H3}
 (\forall\,\delta>0)(\exists\, x^*_\delta>0)\text{ such that }a(x,m)\leq \delta,\quad\forall\, m\in\R_+
\end{equation}
representing that, independently of the level of the network activity, a neuron cannot spike two times in an arbitrary small period of time.  To fix ideas, recall the expression of the refractory intensity function of Pakdaman-Perthame-Salort~\cite{MR2576373}), given by
\begin{equation}\nonumber
\begin{cases}
 a(x,m)=0\text{ for }x\in(0,x^*(m)),
 \\
  a(x,m)>0\text{ for }x>x^*(m),
  \\
 \frac{d}{dm}x^*(m)\,<\,0,\quad x^*(0)\,=\,x^*_+,\quad x^*(\infty)=x^*_->0,
\end{cases}
\end{equation}
where $x^*(m)$ represents a refractory period, $x^*_+$ the spontaneous activity due to noise, and $x^*_-$ the minimal age necessary to spike.

For the previous setting, we have directly the

\begin{prop}\label{prop:syspar-WP}
 Under hypotheses~\eqref{eq:syspar-H1} and~\eqref{eq:syspar-H3}, let $N\geq1$ be fixed and assume that a.s.,
 $$
 \max_{1\leq i\leq N}X_0^{i,N}<\infty,\qquad M^N_0<\infty,
 $$ 
 then there exists a unique \textit{c\`adl\`ag} adapted strong $\R_+$-valued solution $(X^{N}_t,M^N_t)_{t\geq0}$ to~\eqref{eq:syspar-X}-\eqref{eq:syspar-M}.
\end{prop}


Under suitable conditions (to be preciser later on) we show that the solution $(X_t^N)_{t\geq0}$ behave, for large values of $N$, as $N$ independent copies of the solution to a \textit{nonlinear} SDE. Let $Y_0$ (respectively $M_0$) be a $g_0$-distributed random variable (resp. $m_0$) and $\NN_t$ a standard Poisson process independent of $Y_0$ and $M_0$. Then we look for $\R_+$-valued \textit{c\`adl\`ag} adapted process $(Y_t,M_t)_{t\geq0}$ solving for any $t\geq0$
 \begin{equation}\label{eq:meanfield-Y}
  Y_t = Y_0+t- \int_0^tY_{s-}\int_{0}^\infty \indiq_{\{u\leq a(Y_{s-},M_{s-})\}}\,\NN(du,ds),
\end{equation}
and
 \begin{equation}\label{eq:meanfield-M}
  M_t = M_0-\alpha\Big[\int_0^t M_s\,ds -\eps\int_0^t\int_0^s\Ee[a(Y_{s-w},M_{s-w})]\,b(dw)ds\Big].
\end{equation}

\begin{rem}
Let us assume for a moment the hypothesis of instantaneous membrane decay, i.e., $\alpha$ is going to infinity. In that case equation~\eqref{eq:meanfield-M} writes
$$
  M_t\,\,=\,\,\eps\int_0^t\Ee\big[ a( X_{t-w}, M_{t-w})\big]\,b(dw),
$$
in particular, $M_t$ is a deterministic function of $t$. Then, if the probability density of $X_t$ can be written as $f(t,x)\,dx$, the previous relation is reduced to
$$
 M(t) \,\,=\,\,\eps\int_{0}^t\int_0^\infty a(x, M(t-w))\,f(t-w,x)\,dx\,b(dw).
$$
Coming back to equation~\eqref{eq:meanfield-Y}, we have that 
$f(t,x)$ solves (at least in the weak sense) the PDE
 \begin{equation}\nonumber
 \begin{cases}
 \displaystyle\phantom{\Big[}\frac{\partial f}{\partial t} +\frac{\partial f}{\partial x}+a(x,M(t))f(t,x)\,=\,0,
  \\
 \phantom{\Big[}N(t)\,:=\,f(t,x=0) =\int_{0}^\infty a(x, M(t))\,f(t,x)\,dx,
 \end{cases}
 \end{equation}
which is exactly the model that motivates our study.
\end{rem}

The nonlinear SDE is clearly well-posed if we, for instance, make a Lipschitz continuity assumption on the intensity function. In order to avoid this simplification, here we try to use the approaches of Fournier-L{\"o}cherbach~\cite{Fournier:2014nr} and/or Robert-Touboul~\cite{Robert:2014jk}. Nevertheless, we are not able to provide a very sharp result and we have to restrict ourselves to the case of bounded exponential moments. We introduce the quantity
\begin{equation}\nonumber
 \|U_t\|_T\,=\,\sup\{|U_t|,\,\,0\leq t\leq T\},
\end{equation}
which is well defined for locally bounded processes. The second natural result of the manuscript is

\begin{theo}\label{th:meanfield-Wellpossed}
 Let us assume that hypotheses~\eqref{eq:syspar-H1}-\eqref{eq:syspar-H3} hold, then there exists a weak solution $(Y_t,M_t)_{t\geq0}$ to~\eqref{eq:meanfield-Y}-\eqref{eq:meanfield-M} such that
\begin{equation}\label{eq:delay}
 \int_0^t\int_0^s\Ee\big[a(Y_{s-w},M_{s-w})\big]\,b(dw)\,ds<\infty,\qquad \forall\, t\geq0.
\end{equation}
  Moreover, if the law of $(Y_0,M_0)$ is compactly supported, then there exists a unique strong solution $(Y_t,M_t)_{t\geq0}$ to~\eqref{eq:meanfield-Y}-\eqref{eq:meanfield-M} in the class of functions such that there are deterministic locally bounded functions $A,B:\R_+\mapsto\R_+$ such that a.s.
 \begin{equation}\label{eq:locallyHyp}
 \|Y_t/A(t)\|_T<\infty,\quad \|M_t/B(t)\|_T<\infty,\qquad \forall T\,\geq0.
 \end{equation}
 \end{theo}
 \bigskip
 
 The uniqueness result still hold true if the initial datum has a fast decay at infinite. More precisely, let us consider the growing restriction
\begin{equation}\label{eq:polyn}
 (\exists\,\xi>2)\,(\exists\,0<\rho<1)\,(\exists\,C_\xi,c_\rho>0)\,:\,c_\rho\, x^{\frac{1+\rho}{1-\rho}}\leq a(x,m)\leq C_\xi(1+x^{\xi-2}+m^{\xi-2}),
\end{equation}
 and suppose that initial condition has bounded exponential moments
 \begin{equation}\label{eq:Y0expomoments}
  \Ee\big[e^{\omega( Y^{\xi}+M^{\xi}  )}\big]<\infty,\qquad\omega>0.
 \end{equation}
   
 \begin{theo}\label{th:meanfield-Wellpossed2}
 There exists a unique strong solution $(Y_t,M_t)_{t\geq0}$ to~\eqref{eq:meanfield-Y}-\eqref{eq:meanfield-M} in the class of functions of locally bounded exponential moments~\eqref{eq:Y0expomoments}.
\end{theo}


Finally we analyse the chaoticity of the system. To do so, a few more notations must be introduced. We denote by $\mathbb D(\R_+^2)$ the set of c\`adl\`ag functions on $\R_+^2$ endowed with the topology of the convergence on compact time intervals. By definition, each pair $(X_t^{i,N},M_t^N)_{t\geq0}$ belongs to $\mathbb D\mathbb(\R_+^2)$, and then the sequence of empirical measures 
  \begin{equation}\nonumber
  \mu_N=N^{-1}\sum_{i=1}^N\delta_{\{(X_t^{i,N},M_t^N)_{t\geq0}\}},
  \end{equation}
  is well defined and belongs to $\Pp(\mathbb D(\R_+^2))$. 
  
  Note that we intentionally use a two dimensional variable in the definition of $\mu_N$, by repeating $M^N$ in each pair which provides one way to deal with the issue of exchangeability\footnote{One could think of other ways to deal with this issue. For instance, we could define a set of auxiliary variables $M^{i,N}$ each of them solving the equation~\eqref{eq:syspar-M} with different i.i.d. initial conditions.}. The third and last main result of the manuscript is
\begin{theo}\label{th:chaoschaos}
 Let us assume that hypotheses~\eqref{eq:syspar-H1}-\eqref{eq:syspar-H3} hold, and that the law of $(Y_0,M_0)$ is compactly supported, then the sequence of empirical processes $\mu_N(t)$ converges in distribution to the law of the unique process $(Y_t,M_t)_{t\geq0}$ with $(g_0,m_0)$-chaotic initial states solution to~\eqref{eq:meanfield-Y}-\eqref{eq:meanfield-M}. 
 
 If the initial datum has a fast decay (in the sense described in Theorem~\ref{th:meanfield-Wellpossed2}), and if moreover there is a positive constant $C_0$ such that
 \begin{equation}\label{eq:syspar-H4}
 |a(x,m)-a(x',m')|\,\leq\, C_0\big[a(x,m)\wedge a(x',m')\,|x-x'|+|m-m'|\big],
\end{equation}
for all $x,x',m,m'\in\R_+$. Then the convergence of $\mu_N(t)$ remains true. 

In the weak connectivity case, i.e. $\eps\in[0,\eps_0)$ for $\eps_0$ small enough, hypothesis~\eqref{eq:syspar-H4} can be replaced by 
\begin{equation}\label{eq:syspar-H4prima}
 |a(x,m)-a(x+h,m+h)|\,\leq\, C_0\,a(x,m)\,h,
\end{equation}
for all $x,m\in\R_+$ and any $h\in[0,1]$.
\end{theo}

\begin{rem}\
\begin{itemize}
\item Let us notice that condition~\eqref{eq:syspar-H4} is relatively restrictive: indeed, we are asking that the rate function $a$ to be Lipschitz on its second variable, and that the fluctuations on the first variable are somehow bounded. Indeed, a direct consequence of~\eqref{eq:syspar-H4} is that
\begin{equation}\nonumber
 \frac{\partial a(x,m)}{\partial x}=\lim_{h\rightarrow0^+} \frac{a(x+h,m)-a(x,m)}{h} \leq C_0\,a(x,m).
\end{equation}
\item The model is conceived having in mind an intensity function of type
\begin{equation}\nonumber
 a(x,m)\,=\, x^\xi+\indiq_{\{x^*_-<x\}} (am+b)\,\quad \xi\in\N,\quad a,b\in\R_+,
\end{equation}
that formally speaking makes all hypotheses true.
\end{itemize}
\end{rem}

\subsection*{Mathematical overview.} As we already said, the aim of the present work is to give a new microscopic point of view of the age structured equation considered in Pakdaman-Perthame-Salort~\cite{MR2576373,MR3071416,MR3246939}. Therein, the model is proposed as a reinterpretation of the well known renewal equation and the microscopic derivation is omitted. In Tanabe-Pakdaman~\cite{Tanabe:2001vn} and Vibert-Champagnat-Pakdaman-Pham~\cite{Vibert:1998fj} authors propose a particle system but the question of convergence and chaos propagation is not addressed either. Nevertheless, the questions of existence of stationary solutions for the PDE and the numerical/simulation aspects of both: \textit{the particle system and the limit equation}, are deeply studied and several very interesting results, regarding the existence of oscillatory solutions, are given. Moreover, the effects of the finite size of the populations are contrasted with the solutions of the limit equation.

The specific mathematical tools used in the present work can be easily traced down to two recent manuscripts addressing the question of chaoticity of a unidimensional model: Founier-L\"ocherbach~\cite{Fournier:2014nr} and Robert-Touboul~\cite{Robert:2014jk}. The first paper solves the problem under the merely assumption of integrability on the initial condition, which is a remarkable weak hypothesis. Nevertheless, the path-wise uniqueness proof (which is at the end the key point of the method) uses a particular distance that is closely related to the equation itself. A different approach, based on the discretization of the limit equation, is presented in~\cite{de2014hydrodynamic}. There, the convergence is proved imposing compactness on the support of the initial conditions. In~\cite{Robert:2014jk}, the authors also assume boundedness of the initial conditions, and provide a qualitative characterization of the qualitative properties of stationary solutions and their stability in the finite-size and mean-field systems, and a detailed discussion of bifurcations, stability and multiple stationary solutions is given.

The specific mathematical tools used in the present work can be easily traced down to two recent manuscripts addressing the question of chaoticity of a unidimensional model: Founier-L\"ocherbach~\cite{Fournier:2014nr} and Robert-Touboul~\cite{Robert:2014jk}. The first paper solves the problem under the merely assumption of integrability on the initial condition, which is a remarkable weak hypothesis in comparison with other models. Nevertheless, the path-wise uniqueness proof (which is at the end the key point of the method) uses a particular distance that is closely related to the equation itself. A different approach, based on the discretisation of the limit equation, is presenting in De Masi-Galves-L\"ocherbach-Presutti~\cite{de2014hydrodynamic}. There, the convergence is proved imposing compactness on the support of the initial conditions. 

To prove the chaos propagation property there are two classical approaches: whether we use the \textit{coupling} method or an abstract \textit{compactness} argument. The coupling is very intuitive and apply in a very wide range of applications. Nevertheless, usually it is assumed that functions involved are Lipschitz continuous (see e.g.~\cite{touboul2014propagation}). Moreover, the method provides the rate of convergence by explicitly estimate the difference between the empirical measures. In Bolley-Ca\~nizo-Carrillo~\cite{Bolley:2011la}, it is proved that the method still apply in the case of locally Lispchitz continuity, but imposing some exponential moment conditions. The second method is much more general and uses a more abstract framework. It was introduced by Sznitman~\cite{sznitman1984equations}, and it is useful to prove also the existence of solutions to the SDE, but do not provide the rate of convergence.

The present model has two main novelties. The first one is that the system is not one-dimensional but add an extra equation for the coupling variable. This issue implies in particular that attempts to use the distance of Founier-L\"ocherbach~\cite{Fournier:2014nr} fail unless a more suitable distance is found. Moreover, the two-dimensional nature of the empirical measures makes that the rate of convergence attained for a $L^1$ Wasserstein distance is lower than $N^{1/2}$. A perspective of the work is to use a combined PDE/SDE approach to find a sharper entropy function that allows us to have the necessary uniqueness result (see e.g. Godinho~\cite{Godinho:2013kq}). The second novelty of the present work is the presence of delays, that is central in the original PPS model. This issue is solved by using the independence of the random environment (see e.g. Touboul~\cite{touboul2012limits,quininao2014limits}), moreover, the extra terms are treated as locally square integrable processes getting the sharp convergence rate expected.

\subsection*{Plan of the paper.} The paper is organised as follows: Section~\ref{sec:Particles} deals with the well-posedness of the particle system by finding some nice \textit{a priori} bounds of the solutions. Section~\ref{sec:Uniqueness} is related to the path-wise uniqueness question of the mean field system. The result we get remains very restrictive and only can be applied for the case of compactly supported initial conditions and/or fast decay at infinite. Well-posedness of the mean field system is studied in Section~\ref{sec:Chaos}, in particular, we use the compactness argument to find that the sequence of empirical measures converges to a weak solution of the SDE. Finally in Section~\ref{sec:MFConv}, we use the \textit{coupling} method to find different rates of convergence. An appendix completes the present work  with some general well known results in stochastic calculus theory that are useful to our developments.

\section{Study of the particle system}
\label{sec:Particles}
\setcounter{equation}{0}
\setcounter{theo}{0}

Throughout the present section we fix the number of neurons $N\geq1$. For $\mu\in\Pp(\mathbb D(\R_+^2))$ and $\varphi\in C(\R_+^2)$, let us denote the duality product,
\begin{equation}\nonumber
 \langle \mu(t), \varphi\rangle\,:=\, \int_{\mathbb D(\R_+^2)}\varphi(\gamma_t,\beta_t)\,\mu(d\gamma,d\beta).
\end{equation}
Let us notice that for any fixed \textit{random environment}, i.e., any realisation of the initial conditions, the Poisson processes and the delays at $t=0$, we can construct explicitly a unique solution to the problem (it suffices to arrange the jumping times and construct the solution by solving the equation in between two consecutive jumps). Therefore there is a unique strong maximal solution $(X_t^{1,N},\ldots,X_t^{N,N},M^N_t)$ defined on a time interval of the type $[0,\tau^*)$, where $\tau^*$ is given by
\begin{equation}\nonumber
\tau^*\,:=\,\inf\big\{t\geq0\,:\,\min(|X_t^{1,N}|,\ldots,|X_t^{N,N}|,|M_t^N|)=\infty\big\}.
\end{equation}

\begin{proof}[Proof of Proposition~\ref{prop:syspar-WP}]
If we are able to prove that $\tau^*=\infty$, then the particle system~\eqref{eq:syspar-M}-\eqref{eq:syspar-X} is \textit{locally} strongly well-posed proving Proposition~\ref{prop:syspar-WP}. Then it suffices to find some nice \textit{a-priori} bounds.
\smallskip

We start by noticing that any solution $(X_t^{1,N},\ldots, X_t^{N,N},M_t^N)_{t\geq0}$ to~\eqref{eq:syspar-X}-\eqref{eq:syspar-M} satisfies a.s.
\begin{equation}\label{eq:bound-X}
 \max_{1\leq i\leq N} X^{i,N}_t\,\leq\, \max X^{i,N}_0+t,\qquad\forall\,t\geq0.
\end{equation}
Moreover, denoting by $\bar X^N_t$ the (empirical) mean of $(X^{i,N}_t)$, it follows that
 \begin{equation}\nonumber
  \bar X^N_t \,=\, \bar X_0^N(0) +t-\frac1N\sum_{i=1}^N\int_0^t X_{s-}^{i,N}\int_0^\infty\indiq_{\{u\leq a(X_{s-}^{i,N},M_{s-}^N)\}}\NN^i(du,ds),
 \end{equation}
 and using that $\bar X^N_t$ is nonnegative,
 \begin{equation}\nonumber
  \frac1N\sum_{i=1}^N\int_0^t X_{s-}^{i,N}\int_0^\infty\indiq_{\{u\leq a(X_{s-}^{i,N},M_{s-}^N)\}}\NN^i(du,ds)\,\leq\, \bar X_0^N(0) +t.
 \end{equation}
 Next, we fix some $\delta>0$, and use the consistency condition~\eqref{eq:syspar-H3} to get that
  \begin{equation}\label{eq:aux1}
   (\exists\, x^*_\delta)\,\,(\forall\, m\in\R_+,x\leq x^*_\delta),\quad a(x,m)\leq\delta.
  \end{equation}
In particular, using that  $x\geq x^*_\delta\,(1-\indiq_{x<x^*_\delta})$, we have
 \begin{multline}\nonumber
  \frac1N\sum_{i=1}^N\int_0^t \int_0^\infty\indiq_{\{u\leq a(X_{s-}^{i,N},M_{s-}^N)\}}\NN^i(du,ds)\,\leq\, \frac{1}{x^*_\delta}(\bar X_0^N +t)
  +\frac1N\sum_{i=1}^N\int_0^t \int_0^\infty\indiq_{\{u\leq a(x_\delta,M_{s-}^N)\}}\NN^i(du,ds),
 \end{multline}
and taking $\delta<1$, we find a positive constant $C_1$ such that 
\begin{equation}\label{eq:bound-Exp}
 \frac1N\sum_{i=1}^N\int_0^t (1+X_{s-}^{i,N})\int_0^\infty\indiq_{\{u\leq a(X_{s-}^{i,N},M_{s-}^N)\}}\NN^i(du,ds)\,\leq\, C_{1}(t+\bar X_0^N) +Z_t^N,\qquad \forall\, t\geq0,
\end{equation}
where
 $$
 Z_t^N:=N^{-1}\sum_{j=1}^N\int_0^t\int_0^\infty\indiq_{\{u\leq 1\}}\,\NN^j(du,ds).
 $$

Finally, from equation~\eqref{eq:syspar-M}
 \begin{equation}\nonumber
 M_t^N \,\leq\, M_0^N+\frac{\alpha\,\eps}{N}\sum_{j=1}^N \int_{0}^{t-\tau_j}\int_{0}^\infty  \indiq_{\{u\leq a(X^{j,N}_{s-},M_{s-}^N)\}}\,\NN^j(du,ds),
 \end{equation}
 but $a$ is nonnegative, therefore
\begin{equation}\nonumber
 M_t^N \,\leq\, M_0^N+\frac{\alpha\,\eps}{N}\sum_{j=1}^N \int_{0}^{t}\int_{0}^\infty  \indiq_{\{u\leq a(X^{j,N}_{s-},M_{s-}^N)\}}\,\NN^j(du,ds)
 \leq M_0^N+ \alpha\,\eps\,\big( C_1(t+\bar X_0^N)+Z_t^N\big),
 \end{equation}
As a consequence, a.s. the following estimate holds
\begin{equation}\label{eq:bound-M}
 M_t^{N}\,\leq\,M_0^N+C_{1}' (t+\bar X_0^N+Z_t^N),\quad\forall t\geq0,
\end{equation}
for another positive constant $C_{1}'$, depending only on $C_1$, $\alpha$ and $\eps$.

Summarizing, inequalities~\eqref{eq:bound-X} and~\eqref{eq:bound-M} implies that a.s. the solutions to~\eqref{eq:syspar-M}-\eqref{eq:syspar-X} are locally bounded, and therefore $\tau^*=\infty$.
\end{proof}

Let us now study the integrability of the intensity rate $a$. To that aim, let us recall some general results of stochastic calculus for jump processes. First, we notice that for any $1\leq i\leq N$ fixed, the particle system equation can be rewritten by 
\begin{equation}\nonumber
  X^{i,N}_t = X_0^{i,N}+t- \int_0^tX^{i,N}_{s-}\int_{0}^\infty \indiq_{\{u\leq a(X^{i,N}_{s-},M_{s-}^N)\}}\,\NN^i(du,ds),
\end{equation}
\begin{multline}\nonumber
  M_t^N = M_0^N-\alpha\int_0^t M_s^N\,ds+ \frac{\alpha\,\eps}{N}\int_{0}^{t}\indiq_{\{s\leq t-\tau_i\}}\int_{0}^\infty  \indiq_{\{u\leq a(X^{i,N}_{s_-},M_{s_-}^N)\}}\,\NN^i(du,ds)
  \\
  +\frac{\alpha\,\eps}{N}\sum_{j\neq i}^N \int_{0}^{t}\int_{0}^\infty  \indiq_{\{s\leq t-\tau_j\}}\indiq_{\{u\leq a(X^{j,N}_{s_-},M_{s_-}^N)\}}\,\NN^j(du,ds),
\end{multline}
here we see clearly that when the process $\NN^i$ has a jump, then both variables are changed at the same time, unless the jump is in the time interval $(t-\tau_i,t]$. Let $f\in C^1(\R_+^2)$ a regular test function, thanks to the previous remark, the It\^o's formula writes
\begin{multline}\label{eq:Ito-gen}
 f(X^{i,N}_t,M^N_t)\,=\,f(X^{i,N}_0,M^N_0)+\int_0^t\partial_x f(X^{i,N}_s,M^N_s)\,ds-\alpha\,\int_0^tM^N_s\partial_m f(X^{i,N}_s,M^N_s)\,ds
 \\
 +\int_0^{t}\big[f\big(0,M^N_s+\frac{\alpha\,\eps}{N}\indiq_{\{s<t-\tau_i\}}\big)-f(X^{i,N}_s,M^N_s)\big]a(X^{i,N}_s,M^N_s)\,ds
  \\
 +\sum_{j\neq i}\int_0^{t-\tau_j}\Big[f\big(X^{i,N}_{s+\tau_j},M^N_{s+\tau_j}+\frac{\alpha\eps}{N}\big)-f(X^{i,N}_{s+\tau_j},M^N_{s+\tau_j})\Big]a(X^{j,N}_{s},M^N_{s})\,ds\,+\MM^{i,N}_{f}(t).
\end{multline}
If we look carefully, in the no delay case ($\tau_i=0$) equation~\eqref{eq:Ito-gen} is coherent with a instantaneous jump in both variables. The respective local martigale $(\MM_{f}^{i,N}(s))$ is therefore defined by
\begin{multline}\label{eq:Ito-Martingale}
\MM^{i,N}_f(t):= \int_0^t\big[f\big(0,M^N_{s-}+\frac{\alpha\,\eps}{N}\indiq_{\{s<t-\tau_i\}}\big)-f(X^{i,N}_{s-},M^N_{s-})\big]
 \\
 \times\Big[\int_0^\infty\indiq_{\{u\leq a(X^{i,N}_{s-},M^N_{s-})\}}\NN^{i}(du,ds)-a(X^{i,N}_s,M^N_s)\,ds\Big]
 \\
 +\sum_{j\neq i}\int_0^{t-\tau_j}\Big[f\big(X^{i,N}_{s_-+\tau_j},M^N_{s_-+\tau_j}+\frac{\alpha\eps}{N}\big)-f(X^{i,N}_{s_-+\tau_j},M^N_{s_-+\tau_j})\Big]
 \\\times\Big[\int_0^\infty\indiq_{\{u\leq a(X^{j,N}_{s-},M^N_{s-})\}}\NN^j(du,ds)-a(X^{j,N}_{s},M^N_{s})\,ds\Big].
\end{multline}
%
%
\begin{lem}\label{lem:apriori-PS2}
If the intensity function $a(\cdot,\cdot)$ satisfies the growing condition~\eqref{eq:syspar-H4} then, any solution $(X_t^{1,N},\ldots, X_t^{N,N},M_t^N)_{t\geq0}$ to~\eqref{eq:syspar-X}-\eqref{eq:syspar-M} with initial data such that
\begin{equation}\label{eq:hypo-bound-XX}
 \max_{1\leq i\leq N}\Ee\big[a(X^{i,N}_0,M^N_0)^q\big]<\infty,
\end{equation}
with $q\in\{1,\ldots,5\}$, satisfies
\begin{equation}\label{eq:bound-XX}
 \Ee\big[\langle\mu_N(t),a^q\rangle\big]\,<\,\infty,\qquad t\geq0.
\end{equation}
In the no delay case, there is $\eps_0$ such that in the weakly connected regime $\eps\in[0,\eps_0)$, the growing condition~\eqref{eq:syspar-H4} can be replaced by~\eqref{eq:syspar-H4prima}.
\end{lem}

\begin{proof}
We start by applying It\^o's formula~\eqref{eq:Ito-gen} to $f(\cdot,\cdot)=a(\cdot,\cdot)^q$ (using a stoping time if necessary) to get
\begin{multline}\nonumber
 a(X^{i,N}_t,M^N_t)^q\,=\, a(X^{i,N}_0,M^N_0)^q
 +q\,\int_0^ta(X^{i,N}_{s},M^N_{s})^{q-1}\,\partial_x a(X^{i,N}_{s},M^N_{s})\,ds
 \\
 -q\,\alpha\,\int_0^tM_{s}^Na(X^{i,N}_{s},M^N_{s})^{q-1}\,\partial_ma(X^{i,N}_{s},M^N_{s})\,ds
 \\
+\int_0^t\big[a\big(0,M^N_{s}+\frac{\alpha\,\eps}{N}\indiq_{\{s\leq t-\tau_i\}}\big)^q-a(X^{i,N}_{s},M^N_{s})^q\big]\, a(X_{s}^{i,N},M^N_{s})\,ds
\\
 +\sum_{j\neq i}\int_{\tau_j}^{t}\Big[a\big(X^{i,N}_{s},M^N_{s}+\frac{\alpha\,\eps}N\big)^q-a(X^{i,N}_{s},M^N_{s})^q\Big]\, a(X_{s-\tau_j}^{j,N},M^N_{s-\tau_j})\,ds+\MM^{i,N}_{a^q}(t).
\end{multline}
Using that $a(0,\cdot)=0$,  and that $a$ is  non-negative and non-decreasing function in both variables, we get that
\begin{multline}\nonumber
 a(X^{i,N}_t,M^N_t)^q\,\leq\, a(X^{i,N}_0,M^N_0)^q + q\,C_0\int_0^ta(X^{i,N}_{s},M^N_{s})^q\,ds
\\
 +q\,C_0\frac{\alpha\,\eps}{N}\sum_{j\neq i}\int_{\tau_j}^{t}a(X^{i,N}_{s},M^N_{s})^{q-1}a(X_{s-\tau_j}^{j,N},M^N_{s-\tau_j})\,ds+\MM^{i,N}_{a^q}(t),
\end{multline}
implying that
\begin{multline}\nonumber
 a(X^{i,N}_t,M^N_t)^q\,\leq\, a(X^{i,N}_0,M^N_0)^q + q\,C_0\int_0^ta(X^{i,N}_{s},M^N_{s})^q\,ds
 -\int_0^ta(X^{i,N}_{s},M^N_{s})^{q+1}\,ds
\\
 +q\,C_0\alpha\,\eps\int_{0}^{t}a(X^{i,N}_{s},M^N_{s})^{q}\,ds+q\,C_0\frac{\alpha\,\eps}{N}\sum_{j\neq i}\int_{0}^{t}a(X_{s}^{j,N},M^N_{s})^q\,ds+\MM^{i,N}_{a^q}(t),
\end{multline}
Next, we multiply by $N^{-1}$ and add over $i\in\{1,\ldots,N\}$, to get
\begin{multline}\nonumber
\langle\mu_N(t),a^{q}\rangle\leq \langle\mu_N(0),a^{q}\rangle
 +q\,C_0(1+\eps)\int_0^t\langle\mu_N(s),a^{q}\rangle\,ds
 \\
 +q\,C_{0}\alpha\eps\int_0^{t}\langle\mu_N(s),a^{q}\rangle\,ds +\frac1N\sum_{i=1}^N\MM^{i,N}_{a^q}(s),
\end{multline}
taking expectation and thanks to Gronwall's lemma we get the conclusion.\smallskip

In the case of~\eqref{eq:syspar-H4prima}, 
coming back to It\^o's formula~\eqref{eq:Ito-gen}, we get
\begin{multline}\nonumber
a(X^{i,N}_t,M^N_t)^q\,=\, a(X^{i,N}_0,M^N_0)^q
 +q\,C_0\,\int_0^ta(X^{i,N}_{s},M^N_{s})^q\,ds
 -\int_0^ta(X^{i,N}_{s},M^N_{s})^{q+1}\,ds
\\
 +q\,C_0\sum_{j=1}^N\int_0^{t}\frac{\alpha\,\eps}{N}a(X^{i,N}_{s},M^N_{s})^q\, a(X_{s-\tau_j}^{j,N},M^N_{s-\tau_j})\,ds+\MM^{i,N}_{a^q}(t),
\end{multline}
therefore
\begin{multline}\nonumber
\langle\mu_N(t),a^q\rangle\,=\, \langle\mu_N(0), a^q\rangle
 +q\,C_0\,\int_0^t\langle\mu_N(s),a^q\rangle\,ds
 -\int_0^t\langle\mu_N(s),a^{q+1}\rangle\,ds
\\
 +2\,q\,C_0\alpha\,\eps\int_0^{t}\langle\mu_N(s),a^{q+1}\rangle\,ds+N^{-1}\sum_{i=1}^N\MM^{i,N}_{a^q}(t),
\end{multline}
taking expectation we see that the result holds as soon as $1>2\,q\,C_0\alpha\,\eps$.
\end{proof}

%

\section{Path-wise uniqueness of the mean-field system}
\label{sec:Uniqueness}
\setcounter{equation}{0}
\setcounter{theo}{0}

The object of this brief section is to prove under what circumstances the mean field equations~\eqref{eq:meanfield-Y}-\eqref{eq:meanfield-M} are well posed. We restrict our analysis to a very strong hypothesis that englobes many of the applications that can be considered. 

We start by stating some equivalent upper bounds on the solution to the limit equation~\eqref{eq:syspar-X}-\eqref{eq:syspar-M} for the mean field system. The proof is very similar to the arguments used to get~\eqref{eq:bound-X} and~\eqref{eq:bound-M} and therefore we do not go into full details.

\begin{lem}
Any solution $(M_t,Y_t)_{t\geq0}$ to~\eqref{eq:syspar-X}-\eqref{eq:syspar-M} satisfies a.s.
\begin{equation}\label{eq:bound-Y}
 Y_t\,\leq\,Y_0+t,\qquad \forall\, t\geq0.
\end{equation}
Moreover, there is a positive constant $C_2$, depending only on the parameters of the system, such that
\begin{equation}\label{eq:bound-Sum}
 \int_0^t \Ee[(1+Y_{s})\,a(Y_{s},M_{s})]\,ds\,\leq\,C_2(\Ee[Y_0]+t),\qquad \forall\, t\geq0.
\end{equation}
As a consequence, there is another positive constant $C_2'$, such that a.s.
\begin{equation}\label{eq:bound-Mf}
 M_t\,\leq\,M_0+C_2' (t+\Ee[Y_0]),\qquad \forall\, t\geq0.
\end{equation}
\end{lem}
\begin{proof}
 Inequalities~\eqref{eq:bound-Y} and~\eqref{eq:bound-Sum} are easily checked by recalling~\eqref{eq:syspar-X}. The last one, is a consequence of a change of variables. Indeed, we have that
 \begin{multline}\nonumber
  \int_0^t\int_0^s\Ee[a(Y_{s-w},M_{s-w})]\,b(dw)\,ds
  = \int_0^{t}b(dw)\int_{w}^{t}\Ee[a(Y_{s-w},M_{s-w})]\,ds\\
   \leq \int_0^{t}b(dw)\int_{0}^{t}\Ee[a(Y_{s},M_{s})]\,ds.
 \end{multline}
 Finally, we use again that $y\geq y_\delta(1-\indiq_{\{y\leq y_\delta\}})$ to get
 \begin{equation}\nonumber
 \int_0^t \Ee[a(Y_{s},M_{s})]\,ds\,\leq\,y_\delta^{-1}(\Ee[Y_0]+t)+\int_0^t\Ee[a(y_\delta,M_s)]\,ds,
\end{equation}
and the conclusion follows.
\end{proof}

\begin{prop}\label{prop:UniqBounded}
 Path-wise uniqueness holds true for the mean field system~\eqref{eq:meanfield-Y}-\eqref{eq:meanfield-M}, in the class of processes $(Y_t,M_t)_{t\geq0}$ such that there exist deterministic locally bounded functions $A,B:\R_+\mapsto\R_+$ such that a.s.~\eqref{eq:locallyHyp} holds.
\end{prop}

We remark that thanks to inequalities~\eqref{eq:bound-Y} and~\eqref{eq:bound-Mf} we have that if initial conditions are compactly supported, then there are indeed some deterministic locally bounded functions $A,B$ as in the previous proposition.

\begin{proof}
For any two solutions $(Y_t,M_t)_{t\geq0}$ and $(Y'_t,M'_t)_{t\geq0}$, driven by the same Poisson measure $\NN$ and identical initial conditions $(Y_0,M_0)=(Y_0',M_0')$, it holds
\begin{eqnarray*}\nonumber
 \Ee\big[|M_t - M_t'|\big] &\leq& \alpha\,\eps\int_0^t\int_0^s\Ee[|a(Y_{s-w},M_{s-w})-a(Y'_{s-w},M'_{s-w})|]\,b(dw)ds
 \\
 &\leq&\alpha\,\eps\int_0^t \Ee[|a(Y_{s},M_{s})-a(Y'_{s},M'_{s})|]\,ds
 \end{eqnarray*}
 and
 \begin{equation}\nonumber
 \Ee\big[|Y_t - Y_t'|\big] \leq 
\int_0^t \Ee\big[|Y_{s}+Y_{s}'|\,|a(Y_{s},M_{s})-a(Y'_{s},M'_{s})|\big]\,ds.
 \end{equation}
 
 We know that $a$ is a regular differentiable function, therefore it is Lipschitz continuous and bounded on compacts. Since both coordinates are bounded for some deterministic $(A(t),B(t))$ locally bounded functions, it follows that there exists $C_T$ a constant depending only on an upper bound of $A$ and $B$ and the time horizon $T>0$ such that
\begin{equation}\nonumber
 \Ee\big[|Y_t-Y'_t|+|M_t-M'_t|\big]\,\leq\, C_T\int_0^t\Ee\big[|Y_s-Y'_s|+|M_s-M'_s|\big]\,ds,\qquad\forall\,0\leq t \leq T,
\end{equation}
 and the conclusion follows by applying Gr\"onwall's lemma.
\end{proof}

These strong estimates are actually a consequence of the fast that assuming bounded initial conditions essentially reduces the system to considering Lipschitz continuous intensity rate. We can reduce these assumptions by asking fast decay on the initial conditions. More precisely, since
\begin{equation}\nonumber
 Y_t\leq Y_0+t\leq R\quad\Rightarrow\quad \{Y_0+t\leq R\}\subset\{Y\leq R\}\quad\Rightarrow\quad \Pp(Y_0+t\leq R)\leq\Pp(Y_t\leq R),
\end{equation}
then, if initial conditions have an fast decay at infinity as~\eqref{eq:Y0expomoments}, then
\begin{equation}\nonumber
 \Pp\big(Y_0\geq R\big)\,=\, \Pp\Big(e^{\omega Y_0^\xi}\geq e^{\omega R^\xi}\Big)\,\leq\,e^{-\omega R^\xi}\Ee\big[e^{\omega Y_0^\xi}\big]\,\leq\,C_\omega\,e^{-\omega R^\xi},
  \end{equation}
for some $\omega$ positive and $\xi$ given by~\eqref{eq:polyn}, it follows that
\begin{equation}\nonumber
 \Pp(Y_t\geq R)\,\leq\,\Pp\big(Y_0\geq R-t\big)\,\leq\, C_\omega\, e^{-\omega(R-t)^\xi}.
\end{equation}
and an equivalent inequality holds true for $M_t$,
\begin{equation}\nonumber
 \Pp(M_t\geq R)\,\leq\,\Pp\big(M_0\geq R-C_2'(t+\Ee[Y_0])\big)\,\leq\, C_\omega\, e^{-\omega(R-C_2'(t+\Ee[Y_0]))^\xi}.
\end{equation}

\begin{prop}\label{prop:UniqBounded2}
 Under mean field condition~\eqref{eq:syspar-H4} and growth restriction~\eqref{eq:polyn}, path-wise uniqueness holds true for the mean field system~\eqref{eq:meanfield-Y}-\eqref{eq:meanfield-M}, in the class of processes $(Y_t,M_t)_{t\geq0}$ such that for any $t\geq0$ it holds
 \begin{equation}\nonumber
  \sup_{0\leq s\leq t}\Big\{\Ee\big[e^{\omega (Y_s^\xi+M_s^\xi)}\big]\Big\}<\infty.
 \end{equation}
\end{prop}

\begin{proof}
 We start by noticing that since exponential moments are bounded, then all polynomial moments are bounded. In particular,
 \begin{equation}\nonumber
  \Ee\big[a(Y_t,M_t)^4\big]\leq \Ee\big[C_\xi^4(1+Y_t^{\xi+2}+M_t^{\xi+2})^4\big]<\infty,
 \end{equation}
 for any time $t\geq0$. Thanks to hypothesis~\eqref{eq:syspar-H4}, we also have that
 \begin{equation}\nonumber
\big|a(Y_{s},M_{s})-a(Y_{s}',M_{s}')|\leq
C_0\,\big(a(Y_{s},M_{s})+a(Y_{s}',M_{s}')\big)|Y_{s}-Y_{s}'\big|+C_0|M_{s}-M_{s}'|.
\end{equation}

Consider now 
$$\EE(Y,M)=\big\{\sup_{0\leq s\leq t}Y_s\leq R,\quad\sup_{0\leq s\leq t}M_s\leq R\big\},$$
 then
 \begin{multline}\nonumber
  \Ee\big[|Y_t - Y_t'|\big] \leq 
4\,C_0\,R\,a(R,R)\int_0^t \Ee\big[|Y_{s}-Y_{s}'\big|\big]\,ds+2\,C_0\,R\int_0^t \Ee\big[|M_{s}-M_{s}'|\big]\,ds
\\
+\int_0^t \Ee\big[(Y_s+Y_{s}')^4]^{1/4}\Ee\big[|a(Y_{s},M_{s})-a(Y'_{s},M'_{s})|^4\big]^{1/4}\big]\Pp(\EE^c(Y,M))^{1/4}\Pp(\EE^c(Y',M'))^{1/4}\,ds,
 \end{multline}
 and
 \begin{multline}\nonumber
 \Ee\big[|M_t - M_t'|\big]  \leq  
2\,\alpha\,\eps \,C_0\,a(R,R)\int_0^t  \Ee\big[|Y_{s'}-Y_{s'}'\big|\big]\,ds
+\alpha\,\eps\,C_0\int_0^t \Ee\big[|M_{s'}-M_{s'}'|\big]\,ds
\\
 +\alpha\,\eps\int_0^t \Ee\big[|a(Y_{s'},M_{s'})-a(Y'_{s'},M'_{s'})|^2\big]^{1/2}\Pp(\EE^c(Y,M))^{1/4}\Pp(\EE^c(Y',M'))^{1/4}\,ds.
 \end{multline}

Using Gronwall's lemma, we conclude that there exists a constant $C_T$ depending only on the parameters of the system, such that for $R$ large enough
 \begin{equation}\nonumber
  \Ee\big[|Y_{t}-Y_{t}'\big|+|M_{t}-M_{t}'\big|\big]\,\leq\, C_Te^{C_TRa(R,R)}\,\Pp(\EE^c(Y,M))^{1/4}\Pp(\EE^c(Y',M'))^{1/4},
 \end{equation}
 but using the fast decay at infinite
 \begin{equation}\nonumber
\Ee\big[|Y_{s}-Y_{s}'\big|+|M_{s}-M_{s}'\big|\big]\,\leq\, C_T e^{C_TRa(R,R)} e^{-\frac\omega2 R^\xi},
 \end{equation}
 where the constant depends on the time horizon $T$, but not on $R$. Finally thanks to hypothesis~\eqref{eq:polyn}, we get the conclusion:
\begin{equation}\nonumber
 \exp\Big(C_TR\,a(R,R)-\frac\omega2 R^\xi\Big)\xrightarrow{R\rightarrow\infty}0.
 \end{equation}
\end{proof}

\section{The mean field system \& the compactness method}
\label{sec:Chaos}
\setcounter{equation}{0}
\setcounter{theo}{0}

So far, we know that under relatively weak assumptions on $a(\cdot,\cdot)$, for each $N\geq1$, and $(g_0,m_0)$-\textit{chaotic initial states}, there exists a unique solution to~\eqref{eq:syspar-X}-\eqref{eq:syspar-M}. Now we study the convergence of the this set of solutions as $N$ goes to infinity, i.e. the tightness of the sequence of empirical measures $\mu_N$. To that aim, we start by recalling (see e.g. Jacob-Shiryaev~\cite[Theorem 4.5, page 356]{MR1943877}): 
\smallskip

\noindent\textbf{Aldous tightness criterion:}
\textit{the sequence of adapted processes $(X^{1,N}_t,M^N_t)$ is tight if}
 \begin{enumerate}
  \item \textit{for all $T>0$, all $\epsilon>0$, it holds}
  \begin{equation}\nonumber
  \lim_{\delta\rightarrow0^+}\,\limsup_{N\rightarrow\infty}\sup_{(S,S')\in A_{\delta,T}}\Pp\big(|M_S^{N}-M_{S'}^{N}|+|X_S^{1,N}-X_{S'}^{1,N}|>\epsilon\big)=0;
  \end{equation}
  \textit{where $A_{\delta,T}$ is the set of stopping times $(S,S')$ such that $0\leq S\leq S'\leq S+\delta\leq T$ a.s., and}
  \item \textit{for all $T>0$, }
  \begin{equation}\nonumber
  \lim_{K\rightarrow\infty}\sup_{N\geq1}\,\Pp\Big(\sup_{t\in[0,T]}(M_t^N+X_t^{1,N})\geq K\Big)=0.
  \end{equation}
 \end{enumerate}

\begin{prop}\label{prop:meanfield-Tight}
 Under hypothesis~\eqref{eq:syspar-H1} and~\eqref{eq:syspar-H3}. Consider two probability distributions $g_0,m_0$ such that
 \begin{equation}\nonumber
  \int_{\R_+^2}\, (x^2+m^2+a(x,m)^2)g_0(dx)\,m_0(dx)\,<\,\infty,
 \end{equation}
 and the correspondent family $(X^{1,N}_t,\ldots,X^{N,N}_t,M^N_t)_{t\geq0}$ of solutions to~\eqref{eq:syspar-X}-\eqref{eq:syspar-M}, starting with some i.i.d. random variables $X_0^{i,N}$ with common law $g_0$, and independent of the $m_0$-distributed $M_0^N$. Then 
 \begin{itemize}
  \item[(i)] the sequence of processes $(X_t^{1,N},M^N_t)_{t\geq0}$ is tight in $\mathbb D(\R_+^2)$;
  \item[(ii)] the sequence of empirical measures $\mu_N$ is tight in $\Pp\big(\mathbb D(\R_+^2)\big)$.
 \end{itemize}
\end{prop}

Let us remark that the sequence $Z^i:=(X_t^{i,N},M_t^N)$ is exchangeable, then (ii) follows from (i) thanks to Sznitman~\cite[Proposition 2.2-(ii)]{MR1108185}).

\begin{proof}
 We only need to show the Aldous tightness criterion, let us notice that the second condition is easy to show. Indeed, from estimate~\eqref{eq:bound-X} we get that
 \begin{equation}\nonumber
 \Ee\big[\sup_{0\leq s\leq T} X_s^{1,N}\big]\,\leq \Ee[X_0^{1,N}]+T\,\leq\,\Big(\int_{\R_+}\, x^2\,g_0(dx)\Big)^{1/2}+T\,<\,\infty,
 \end{equation}
 and recalling~\eqref{eq:bound-M}, we notice that
 \begin{equation}\nonumber
 \Ee\big[\sup_{t\in[0,T]}M_t^{N}\big]\,\leq\,\Ee[M_0^N]+C_1' (T+\Ee[\bar X_0^N]+\Ee[Z_t^N])\,<\,\infty,
\end{equation}
because $Z^N_t$ is the mean of $N$ i.i.d Poisson($T$)-distributed random variable. We deduce that the expectation of the lefthand side is finite, independently of the value of $N$.
\smallskip

Under the assumptions on the initial distributions, the first point is not very difficult to prove either. Indeed, by definition
\begin{equation}\nonumber
 |X_S^{1,N}-X_{S'}^{1,N}|\,\leq\,(S'-S)+\int_{S}^{S'}\int_0^\infty X_{s-}^{1,N}\indiq_{\{u\leq a(X_{s-}^{1,N},M_{s-}^N)\}}\NN^1(du,ds),
\end{equation}
and
\begin{equation}\nonumber
 |M^N_S-M^N_{S'}|\,\leq\,\frac{\alpha\,\eps}{N}\sum_{j=1}^N\int_{S}^{S'}\int_0^\infty\indiq_{u\leq a(X_{s-\tau_j}^{j,N},M^N_{s-\tau_j})}\,\tilde\NN^j(du,ds)+\alpha\int_{S}^{S'}M_s^N\,ds.
\end{equation}
Moreover, using Markov's inequality
\begin{multline}\nonumber
 \Pp\Big(\int_{S}^{S'}\int_0^\infty X_{s-}^{1,N}\indiq_{\{u\leq a(X_{s-}^{1,N},M_{s-}^N)\}}\NN^1(du,ds)>0\Big) \,\leq\, \Ee\Big[\int_S^{S+\delta}a(X_{s}^{1,N},M_{s}^N)\,ds\Big]
 \\
 \leq\, \delta^{1/2}\times\Ee\Big[\Big(\int_0^{T}a(X_{s}^{1,N},M_{s}^N)^2\,ds\Big)^{1/2}\Big],
\end{multline}
which is finite independently of the value of $N$ thanks to Lemma~\ref{lem:apriori-PS2}. The first term in the second inequality is handled in the same way. Finally,
\begin{equation}\nonumber
 \Ee\Big[\int_S^{S'}M^N_s\,ds\Big]\,\leq\,\delta^{1/2}\,\Ee\Big[\Big(\int_0^{T}(M^N_s)^2\,ds\Big)^{1/2}\Big],
\end{equation}
which is also finite thanks to the assumption on the initial condition.
\smallskip

\end{proof}

The natural next step, in the proof of existence of solutions to the nonlinear SDE, is to prove that any limit point of the tight sequence $\mu_N$ is a solution of the mean field limit system, which is usually called consistency of the particle system. This result is stated in the following

\begin{prop}\label{prop:meanfield-Sol}
 Under the same hypotheses of Proposition~\ref{prop:meanfield-Tight}, any limit point $\mu$ of $\mu_N$ a.s. belongs to 
 \begin{multline}\nonumber
 \SS:=\Big\{\LL((Y_t,M_t)_{t\geq0})\,:\,(Y_t,M_t)_{t\geq0}\text{ is a solution to~\eqref{eq:meanfield-Y}-\eqref{eq:meanfield-M} such that}
 \\
 \loi(Y_0)=g_0,\quad\loi(M_0)=m_0\,\quad
 \text{and}\quad
 \int_{0}^t\int_0^s \Ee\big[a(Y_{s-w},M_{s-w})\,b(dw)\big]\,ds\,<\,\infty,\,\forall\,t\geq0\Big\}.
 \end{multline}
 \end{prop}

A previous step that simplifies the proof of this result is the
\begin{lem}\label{lem:conditions-T}
  Let us consider $t\geq0$ fixed and define $\pi_t:\mathbb D(\R_+^2)\rightarrow\R_+^2$, by
 \begin{equation}\nonumber
 \pi_t(\gamma,\beta)=(\gamma_t,\beta_t).
 \end{equation}
Then, $Q\in \Pp(\mathbb D(\R_+^2))$ belongs to $\SS$ if the following conditions are satisfied:
 \begin{itemize}
  \item[(a)] $Q\circ\pi^{-1}_0\,=\,(g_0,m_0)$;
  \item[(b)] for all $t\geq0$, $$\int_{\mathbb D(\R_+^2)}\int_0^t\int_0^s  a(\gamma_{s-w},\beta_{s-w})\,b(dw)\,ds\,Q(d\gamma,d\beta)\,<\,\infty;$$
  \item[(c)] for any $0\leq s_1<\ldots<s_k<s<t,$ any $\varphi_1,\ldots,\varphi_k\in C_b(\R_+^2)$, and any $\varphi\in C_b^2(\R_+^2)$, it holds
  \begin{multline}\nonumber
   F(Q)\,:=\,\int_{\mathbb D(\R_+^2)}\int_{\mathbb D(\R_+^2)} Q(d\gamma^1,d\beta^1)\,Q(d\gamma^2,d\beta^2)\,\varphi_1(\gamma^1_{s_1},\beta^1_{s_1})\ldots\varphi_k(\gamma^1_{s_k},\beta^1_{s_k})\\
   \Big[\varphi(\gamma^1_{t},\beta^1_{t})-\varphi(\gamma^1_{s},\beta^1_{s})-\int_s^t\partial_\beta\varphi(\gamma^1_{s'},\beta^1_{s'})\big[-\alpha\beta^1_{s'}+\alpha\,\eps\int_0^{s'} a(\gamma^2_{s'-w},\beta^2_{s'-w})\,b(dw)\big]\,ds'
   \\
   -\int_s^t\partial_\gamma\varphi(\gamma^1_{s'},\beta^1_{s'})\,ds'-\int_s^t a(\gamma^1_{s'},\beta^1_{s'})\big[\varphi(0,\beta^1_{s'})-\varphi(\gamma^1_{s'},\beta^1_{s'})\big]\,ds'\Big]=0.
  \end{multline}
 \end{itemize}
\end{lem}
\begin{proof}
 Let us consider a process $(Y_t,M_t)_{t\geq0}$ of law $Q$ which satisfies (a), (b) and (c). From (a) and the independency of $m_0$ and $g_0$, we have
 \begin{equation}\nonumber
  \LL(Y_0)\,=\,g_0,\quad \LL(M_0)\,=\,m_0,
 \end{equation}
 and from (b) we have that
 \begin{equation}\nonumber
  \int_0^t\int_0^s\Ee\big[a(Y_{s-w},M_{s-w})\big]\,b(dw)\,ds<\infty,\quad\forall\,t\geq0.
 \end{equation}
 Finally, from (c) we have that for any $\varphi\in C_b^2(\R_+^2)$, the process
 \begin{multline}\nonumber
  \varphi(Y_t,M_t)-  \varphi(Y_0,M_0)-\int_0^t\partial_\beta\varphi(Y_s,M_s)\Big[-\alpha M_s+\alpha\,\eps\int_0^s\Ee[a(Y_{s-s'},M_{s-s'})]\,b(ds')\Big]ds
  \\
  -\int_0^t\partial_\gamma\varphi(Y_s,M_s)\,ds-\int_0^t a(Y_s,M_s)\big(\varphi(0,M_s)-\varphi(Y_s,M_s)\big)ds,
 \end{multline}
 is a local martingale. The conclusion follows as an application of Jacob-Shiryaev~\cite[Theorem II.2.42 page 86]{MR1943877} and~\cite[Theorem III.2.26 page 157]{MR1943877}. This result is classic, but for completeness of the present manuscript, we provide some remarks on Appendix~\ref{sec:AppJacob}.
\end{proof}

We finish this section by giving the proof of Proposition~\ref{prop:meanfield-Sol}:

\begin{proof}[Proof of Proposition~\ref{prop:meanfield-Sol}]
At this point, the problem is reduced to prove that conditions (a), (b) and (c) of Lemma~\ref{lem:conditions-T} hold. Since we do not have much information about $\mu$ we cannot work directly with it. On the other hand, we know that $\mu_N$ (up to subsequence) is converging to $\mu$, therefore, it seems natural to use equations~\eqref{eq:syspar-X}-\eqref{eq:syspar-M} adequately and then pass to the limit.\smallskip

\noindent\textit{Step 1.} Let us recall that for any $N\geq1$, the random variables $X_0^{i,N}$ are i.i.d. with common law $g_0$, and that $M^N_0$ is random variable with law $m_0$, independent of $(X_0^{i,N})$. It follows that
\begin{equation}\nonumber
 \mu\circ \pi_0^{-1}\,=\,\lim_{N\rightarrow\infty}\frac1N\sum_{i=1}^N\delta_{(X^{i,N}_0,M_0^N)}\,=\,(g_0,m_0).
\end{equation}
We also have, by the Fatou's lemma and inequality~\eqref{eq:bound-Exp}, that
\begin{multline}\nonumber
\Ee\Big[\int_{\mathbb D(\R_+^2)} \int_{0}^t\int_0^s\,\big[a(\gamma_{s-w},\beta_{s-w})\wedge K\big]\,b(dw)\,ds\,\mu(d\gamma,d\beta)\Big]
\\
\leq\,\liminf_{N\rightarrow\infty}\frac1N\sum_{i=1}^N\,\int_{0}^t\int_0^s\Ee[a(X^{i,N}_{s-w},M^N_{s-w})]\,b(dw)\,ds\,<\,\infty,
\end{multline}
for any $t\geq0$. Letting $K\rightarrow\infty$ we get (b).
\smallskip

\noindent\textit{Step 2.}  It only remains to prove (c), to that aim, we start by noticing that $F(\mu_N)$ writes
  \begin{multline}\nonumber
   F(\mu_N)\,=\,\frac{1}{N}\sum_{i=1}^N\varphi_1(X^{i,N}_{s_1},M^N_{s_1})\ldots\varphi_k(X^{i,N}_{s_k}M^N_{s_k})
   \\
   \Big[\varphi(X^{i,N}_{t}M^N_{t})-\varphi(X^{i,N}_{s},M^N_{s})-\int_s^t\partial_\gamma\varphi(X^{i,N}_{s'},M^N_{s'})\,d{s'}+\alpha\int_s^tM^N_{s'}\partial_\beta\varphi(X^{i,N}_{s'},M^N_{s'})\,d{s'}
   \\
   -\int_s^t \big[\varphi(0,M^N_{s'})-\varphi(X^{i,N}_{s'},M^N_{s'})\big]a(X^{i,N}_{s'},M^N_{s'})\,d{s'}
   \\
   -\int_s^t\partial_\beta\varphi(X^{i,N}_{s'}M^N_{s'})\frac{\alpha\,\eps}{N}\sum_{j=1}^N\int_0^{s'} a(X^{j,N}_{s'-w},M^N_{s'-w})\,b(dw)\,ds'
   \Big].
  \end{multline}
At the same time, using the It\^o's formula~\eqref{eq:Ito-gen} to the test function $\varphi(\cdot,\cdot)$, we have
\begin{multline}\nonumber
 \varphi(X^{i,N}_tM^N_t)\,=\, \varphi(X^{i,N}_0,M^N_0)+\int_0^t\partial_\gamma\varphi(X^{i,N}_{s'},M^N_{s'})\,ds' -\alpha\int_0^tM_{s'}^N\partial_\beta\varphi(X^{i,N}_{s'},M^N_{s'})\,ds'
 \\
+\int_0^t\big[\varphi\Big(0,M^N_{s'}+\frac{\alpha\,\eps}{N}\indiq_{\{s\leq t-\tau_i\}}\Big)-\varphi(X^{i,N}_{s'},M^N_{s'})\big] \int_0^\infty\indiq_{\{u\leq a(X_{s'}^{i,N},M^N_{s'})\}}\NN^i(du,ds')
\\
 +\sum_{j\neq i}\int_0^{t-\tau_j}\big[\varphi(X^{i,N}_{s'+\tau_j},M^N_{s'+\tau_j}+\frac{\alpha\,\eps}N)-\varphi(X^{i,N}_{s'+\tau_j},M^N_{s'+\tau_j})\big] \int_0^\infty\indiq{\{u\leq a(X_{s'}^{j,N},M^N_{s'})\}}\NN^j(du,ds'),
\end{multline}
implying, that $F(\mu_N)$ can be rewritten by
\begin{equation}\nonumber
 F(\mu_N)\,=\,\frac{1}{N}\sum_{i=1}^N\varphi_1(X^{i,N}_{s_1},M^N_{s_1})\ldots\varphi_k(X^{i,N}_{s_k},M^N_{s_k})\big[(R_t^{i,N}-R_s^{i,N})+(\Delta_{t}^{i,N}-\Delta_{s}^{i,N})\big],
\end{equation}
with
\begin{multline}\nonumber
 R^{i,N}_t:=\int_0^t\big[\varphi\Big(0,M^N_{s-}+\frac{\alpha\,\eps}{N}\indiq_{\{s\leq t-\tau_i\}}\Big)-\varphi(X^{i,N}_{s-},M^N_{s-})\big]
 \\
 \times\Big[\int_0^\infty\indiq_{\{u\leq a(X_{s-}^{i,N},M^N_{s-})\}}\NN^i(du,ds)-a(X_{s-}^{i,N},M^N_{s-})\,ds\Big],
\end{multline}
and
\begin{eqnarray*}
 \Delta_{t}^{i,N}&:=& \sum_{j\neq i}\int_0^{t-\tau_j}\int_0^\infty\big[\varphi\Big(X^{i,N}_{s_-+\tau_j},M^N_{s_-+\tau_j}+\frac{\alpha\,\eps}N\Big)-\varphi(X^{i,N}_{s_-+\tau_j},M^N_{s_-+\tau_j})\big]\indiq_{\{u\leq a(X_{s-}^{j,N},M^N_{s-})\}}\NN^j(du,ds)
 \\
 &&\quad-\int_0^t\partial_\beta\varphi(X^{i,N}_{s},M^N_{s})\frac{\alpha\,\eps}{N}\sum_{j=1}^N\int_0^s a(X^{j,N}_{s-w},M^N_{s-w})\,b(dw)\,ds.
\end{eqnarray*}
Using that the Poisson processes $\NN^i$ are i.i.d., we get that the compensated martingales $R^{i,N}_t$ are orthogonal, and thanks to the exchangeability, we get that
\begin{equation}\nonumber
 \Ee[|F(\mu_N)|]\,\leq\, \frac{C_F}{\sqrt{N}}\Ee[(R^{1,N}_t-R^{1,N}_s)^2]^{1/2}+C_F\,\Ee[|\Delta^{1,N}_t|+|\Delta^{1,N}_s|]
\end{equation}
for some positive $C_F$ depending on the upper bounds of the test functions composing $F$. Moreover, the first expectation is bounded uniformly on $N$:
\begin{eqnarray*}
 \Ee\big[(R^{1,N}_t-R^{1,N}_s)^2\big]
\,\leq\, C_F\int_0^t\Ee\big[a(X_{s}^{1,N},M^N_{s})\big]\,ds,
\end{eqnarray*}
which is finite thanks to~\eqref{eq:bound-Exp}.

For the second expectation, we split $\Delta^{1,N}_t$ in four quantities that can be handled separately:
\begin{multline}\nonumber
 |\Delta^{1,N}_t|\leq
 \int_{0}^{t-\tau_1}\big|\varphi(X^{1,N}_{s_-+\tau_1},M^N_{s_-+\tau_1}+\frac{\alpha\,\eps}N)-\varphi(X^{1,N}_{s_-+\tau_1},M^N_{s_-+\tau_1})\big|\indiq_{\{u\leq a(X_{s-}^{1,N},M^N_{s-})\}}\NN^1(du,ds)
 \\
 +\Big|\sum_{j=1}^N\int_{0}^{t-\tau_j}\big[\varphi(X^{1,N}_{s_-+\tau_j},M^N_{s_-+\tau_j}+\frac{\alpha\,\eps}N)-\varphi(X^{1,N}_{s_-+\tau_j},M^N_{s_-+\tau_j})\big]
 \\
 \qquad\qquad\qquad\times\Big[\int_0^\infty\indiq_{\{u\leq a(X_{s-}^{j,N},M^N_{s-})\}}\NN^j(du,ds)-a(X_{s}^{j,N},M^N_{s})\,ds\Big]\Big|
 \\
 +\sum_{j=1}^N\int_{0}^{t-\tau_j}\Big|\big[\varphi(X^{1,N}_{s+\tau_j},M^N_{s+\tau_j}+\frac{\alpha\,\eps}N)-\varphi(X^{1,N}_{s+\tau_j},M^N_{s+\tau_j})-\frac{\alpha\eps}{N}\partial_\beta\varphi(X^{1,N}_{s+\tau_j},M^N_{s+\tau_j})\big]a(X_{s}^{j,N},M^N_{s})\Big|\,ds
 \\
 +C_F\frac{\alpha\,\eps}{N}\Big|\sum_{j=1}^N\Big(\int_{\tau_j}^{t}a(X^{j,N}_{s-\tau_j},M^N_{s-\tau_j})\,ds-\int_0^{t}\int_0^s a(X^{j,N}_{s-w},M^N_{s-w})\,b(dw)\,ds\Big)\,\Big|
 :=\TTT_1+\TTT_2+\TTT_3+\TTT_4.
\end{multline}
The first three terms are controllable simply using that $\varphi\in C^2_b(\R^2_+)$. Indeed, for any $(x,m)\in\R^2_+$, we have that
\begin{equation}\nonumber
 \varphi\Big(x,m+\frac{\alpha\eps}{N}\Big)\,=\,\varphi(x,m)+\frac{\alpha\eps}{N}\,\partial_\beta\varphi(x,m)+\frac12\Big(\frac{\alpha\eps}{N}\Big)^2\,\partial^2_{\beta\beta}\varphi(x,m)+O(N^{-2}),
\end{equation}
using again~\eqref{eq:bound-Exp} we get that the respective expectations are going to 0 when $N$ goes to infinity (using Holder's inequality to find the convergence). 

The contribution of $\TTT_4$ must be handled more carefully, nevertheless, we have that
\begin{equation}\nonumber
 \Ee_{\text{delays}}\Big[\int_{\tau_j}^ta(X^{j,N}_{s-\tau_j},M^N_{s-\tau_j})\,ds\Big]\,\,=\,\,\int_0^\infty\int_w^t a(X^{j,N}_{s-w},M^N_{s-w})\,ds\,b(dw)\,\,=\,\,\int_0^t\int_0^s a(X^{j,N}_{s-w},M^N_{s-w})\,b(dw)\,ds.
\end{equation}
for any $1\leq j\leq N$, therefore, each term defined by
\begin{equation}\nonumber
 \TTT_{j,4}(t)\,:=\int_{\tau_j}^{t}a(X^{j,N}_{s-\tau_j},M^N_{s-\tau_j})\,ds-\int_0^{t}\int_0^s a(X^{j,N}_{s-w},M^N_{s-w})\,b(dw)\,ds,
\end{equation}
has zero expectation. Using that the delays $\tau_j$ are i.i.d, we get that
\begin{equation}\nonumber
\Ee\big[\TTT_{j,4}(t)\times\TTT_{k,4}(t)\big]\,=\,\Ee\big[\Ee_{\text{delays}}[\TTT_{j,4}(s)]\,\times\,\Ee_{\text{delays}}[\TTT_{k,4}(s)]\big]=0,
\end{equation}
if $i\neq k$, then
\begin{equation}\nonumber
 \Ee[\TTT_4]\,\leq\,C_F\,\frac{\alpha\,\eps}{N}\Big(\sum_{j,k=1}^N\Ee\big[\TTT_{j,4}(t)\times\TTT_{k,4}(t)\big]\Big)^{1/2}\,ds\,=\,C_F\,\frac{\alpha\,\eps}{\sqrt{N}}\Big(\Ee\big[(\TTT_{1,4}(t))^2\big]\Big)^{1/2}.
\end{equation}
Finally, using Lemma~\ref{lem:apriori-PS2}, we see that 
$$
\Ee\big[\TTT_{1,4}(s)^2\big]\,\leq\,2\,\Ee\Big[\Big(\int_{\tau_1}^{t}a(X^{1,N}_{s-\tau_1},M^N_{s-\tau_1})\,ds\Big)^2+\Big(\int_0^{t}\int_0^s a(X^{1,N}_{s-w},M^N_{s-w})\,b(dw)\,ds\Big)^2\Big],
$$
 the first quantity is bounded because
 \begin{equation}\nonumber
  \int_{\tau_1}^ta(X^{1,N}_{s-\tau_1},M^N_{s-\tau_1})\,ds\,\leq\, t^{1/2}\times\Big(\int_{0}^ta(X^{1,N}_{s},M^N_{s})^2\,ds\Big)^{1/2},
 \end{equation}
 and the second because of
 \begin{equation}\nonumber
  \int_0^{t}\int_0^s a(X^{1,N}_{s-w},M^N_{s-w})\,b(dw)\,ds\,\leq\,\int_0^{t}a(X^{1,N}_{s},M^N_{s})\,ds,
 \end{equation}
 summarizing, $\TTT_4$ is also going to 0 with $N$.
\bigskip

\noindent\textit{Step 3.} Before passing to the limit we still need to be sure that no mass is added in the discontinuity points of the paths, i.e., we need to check that for any $t\geq0$, a.s., $$\mu(\{(\gamma,\beta)\,:\,\Delta(\gamma,\beta)(t)\neq0\})=0.$$ 

The proof is exactly as in~\cite[Theorem 5-(iii)-Part 2]{Fournier:2014nr} but for completeness we give some remarks. In order to get a contradiction, we assume that there are some $b,d>0$ such that
\begin{equation}\nonumber
 \Pp[E]>0,\quad\text{with}\quad E:=\big\{\mu(\{(\gamma,\beta):\max(|\Delta\gamma(t)|,|\Delta\beta(t)|)>b\})>d\big\}.
\end{equation}
Therefore, for any $\epsilon>0$, it holds
\begin{equation}\nonumber
 E\subset\{\mu(B_b^\epsilon)>d\},\quad B_b^\epsilon:=\{\mu:\sup_{s\in(t-\epsilon,t+\epsilon)}\max(|\Delta\gamma^1(s)|,|\Delta\gamma^2(s)|)>b\}.
\end{equation}
Moreover, $B_b^\epsilon$ is an open subset of $\mathbb D(\R_+^2)$, then $$\PP_{b,d}^\epsilon:=\{Q\in\Pp(\mathbb(\R_+^2):Q(B_b^\epsilon)>d\}\subset\Pp(\mathbb D(\R_+^2)$$ is also an open set. Thanks of Portmanteau theorem we get that for any $\epsilon>0$,
\begin{equation}\nonumber
 \liminf_{N\rightarrow\infty}\, \Pp(\mu_N\in\PP_{b,d}^\epsilon)\geq\Pp(\mu\in\PP_{b,d}^\epsilon)\geq \Pp(E)>0.
\end{equation}

On the other hand, for $N$ large enough, the jumps in equation~\eqref{eq:syspar-M} are smaller than $b$ and then the problem is reduced to control the size of the jumps in equation~\eqref{eq:syspar-X}, and in particular to show that
\begin{equation}\nonumber
 \Pp(\mu_N\in\PP_{b,d}^\eps)\leq \Pp\Big(\frac{1}{N}\sum_{i=1}^N\indiq_{\big\{\int_{t-\epsilon}^{t+\epsilon}\indiq_{\{u\leq a(X^{i,N}_{s-},M^N_{s-})\}}\NN^i(du,ds)\geq1\big\}}\geq b\Big)\rightarrow0,
\end{equation}
which can be easily done using the same arguments of the proof of Proposition~\ref{prop:meanfield-Tight}.
\smallskip

\noindent\textit{Step 4.} Now we see that $F$ is a continuous function at any point $Q\in\Pp(\mathbb D(\R_+^2)$ such that
$$
 Q(\{(\gamma,\beta)\,:\,\Delta(\gamma,\beta)(s_1)=\ldots=\Delta(\gamma,\beta)(s_k)=\Delta(\gamma,\beta)(s)=\Delta(\gamma,\beta)(t)=0\})\,=\,1,
$$
and
$$
 \int_{\mathbb D(\R_+^2)}\int_0^t\int_0^s a(\gamma_{s-w},\beta_{s-w})\,b(dw)\,ds\, Q(d\gamma,d\beta)\,<\,\infty.
$$
Thanks to Step 2 and 3 we know that our limit belongs to this subset of $\Pp(\mathbb D(\R_+^2)$, and therefore
$$
 \Ee\big[|F(\mu)|\big]\,\leq\,\lim_{K\rightarrow\infty}\limsup_{N\rightarrow\infty} \Ee\big[|F(\mu_N)|\wedge K\big]\,=\,0.
$$
\end{proof}

So far we have built a weak solution $(Y_t,M_s)_{t\geq0}$ to the mean field system~\eqref{eq:meanfield-Y}-\eqref{eq:meanfield-M}, such that $\int_0^t\int_0^s\Ee[a(Y_{s-w},M_{s-w})]\,b(dw)\,ds$ is finite for all times. Thanks to the path-wise uniqueness result of Section~\ref{sec:Uniqueness} we can go a little further by providing the

\begin{proof}[Proof of Theorem~\ref{th:meanfield-Wellpossed}]
Proposition~\ref{prop:meanfield-Sol} gave us already the existence of a weak solution such that~\eqref{eq:delay} holds true. Furthermore, if the initial laws $(g_0,m_0)$ are compactly supported, then thanks to inequalities~\eqref{eq:bound-Y} and~\eqref{eq:bound-Mf}, we get that a.s.
 $$
  Y_t\,\leq\, Y_0+t\,\leq\, S_{g_0}+t=: A(t),\qquad M_t\,\leq\,M_0+C_2' (t+\Ee[Y_0])\,\leq\, S_{m_0}+C_2'(t+\Ee[Y_0])=:B(t),
 $$
 with $S_{g_0}$ (respectively $S_{m_0})$ any upper bound of the support of $g_0$ (resp. $m_0$). Proposition~\ref{prop:UniqBounded} implies that there is a unique solution such that the previous conditions hold true, and this solutions is exactly a process with law $\mu$ (it suffices to pass to the limit in the associated particle system).
\end{proof}
 
 The existence of a strong solution in the case of exponential decay of initial conditions is similar and therefore omitted.

\section{Mean-Field convergence by the coupling method}
\label{sec:MFConv}
\setcounter{equation}{0}
\setcounter{theo}{0}

In this final section we use the ideas of coupling to prove a quantified version of the convergence of the empirical laws $\mu_N$ towards the law of the unique process that solves~\eqref{eq:meanfield-Y} and~\eqref{eq:meanfield-M}. We start by noticing that, thanks to Theorem~\ref{th:meanfield-Wellpossed}, there exists a family of stochastic processes 
\begin{equation}\label{eq:setofIC}
(Y^{1,N}_t,M^{i,N}_t,\ldots, Y^{N,N}_t,M^{N,N}_t)_{t\geq0},
\end{equation}
 such that
 \begin{equation}\nonumber
   Y_t^{i,N} = X_0^{i,N}+t- \int_0^tY_{s-}^{i,N}\int_{0}^\infty \indiq_{\{u\leq a(Y_{s-}^{i,N},M_{s-}^{i,N})\}}\,\NN^i(du,ds),
 \end{equation}
 and
 \begin{equation}\nonumber
  M_t^{i,N} = M_0^N-\alpha\Big[\int_0^t M_s^{i,N}\,ds -\eps\int_0^t\int_0^s\Ee[a(Y_{s-w}^{i,N},M_{s-w}^{i,N})]\,b(dw)ds\Big],
\end{equation}
where the initial conditions and the Poisson processes are exactly as described in the introduction. In the following, we use the notation $\eta_N(t)$ for the empirical mean associated to the exchangeable family $(Y^{i,N}(t),M^{i,N}(t))$. 

By the definition of $X^{i,N}_t$, it follows that
 \begin{equation}\nonumber
   X_t^{i,N}-Y_t^{i,N} = \int_0^t\int_{0}^\infty \Big[-X_{s-}^{i,N}\indiq_{\{u\leq a(X_{s-}^{i,N},M_{s-}^N)\}}+ Y_{s-}^{i,N}\indiq_{\{u\leq a(Y_{s-}^{i,N},M_{s-}^{i,N})\}}\Big]\,\NN^i(du,ds),
 \end{equation}
then
 \begin{multline}\label{eq:ModulusXi}
 \Ee\big[|X_t^{i,N}-Y_t^{i,N}|\big]\,\leq\, -\int_0^t\Ee\Big[\big|X_{s}^{i,N}- Y_{s}^{i,N}\big|\,a(X_{s}^{i,N},M_{s}^N)\wedge a(Y_{s}^{i,N},M_{s}^{i,N})\Big]\,ds
 \\
 +\int_0^t\Ee\Big[ \big(X_{s}^{i,N}+ Y_{s}^{i,N}\big) \big|a(X_{s}^{i,N},M_{s}^N)-a(Y_{s}^{i,N},M_{s}^{i,N})\big|\Big]\,ds.
 \end{multline}

Since initial distribution are compactly supported, we notice that $Y_t^{i,N}\leq A(t)$ and $M_t^{i,N}\leq B(t)$ for some locally bounded functions $A,B$ independent of $N$. Therefore, $a$ is bounded and Lipschitz continuous in both variables. Then, for any $i=1,\ldots,N$, we get that
 \begin{equation}\nonumber
  \Ee\big[|X^{i,N}_s-Y^{i,N}_s|\big]\,\leq\, C_T\int_0^t\Ee\big[\big(|X^{i,N}_{s'}-Y^{i,N}_{s'}|+|M^{N}_{s'}-M^{i,N}_{s'}|\big)\big]\,ds,
 \end{equation}
 for some positive constant $C_T$ independent of $N$. Similarly, by the definition of $M^N_t$, we have
 \begin{multline}\label{eq:ModulusMn}
  \frac1N\sum_{i=1}^N \Ee\big[|M^{N}_t-M_t^{i,N}|\big] \,\leq\,   \frac\alpha N\sum_{i=1}^N\int_0^t \Ee\big[|M_s^{N}-M_s^{i,N}|\big]\,ds
   \\
    +\frac{\alpha\eps}{N}\,\sum_{i=1}^N\int_{0}^{t}\Ee\big[|a(X^{i,N}_{s},M_{s}^N)-a(Y^{i,N}_{s},M_{s}^{i,N})|\big]\,ds
\\
    +\frac{\alpha\eps}{N}\Ee\Big[\Big|\sum_{i=1}^N\Big(\int_{0}^{t-\tau_i}a(Y^{i,N}_{s},M_{s}^{i,N})\,ds-\int_0^t\int_0^s\Ee\big[a(Y^{i,N}_{s-w},M^{i,N}_{s-w})\big]\,b(dw)\,ds\Big)\Big|\Big],
 \end{multline}
 we finally see that using the arguments of \textit{Step 2} of the proof of Proposition~\ref{prop:meanfield-Sol}, we notice that there is another positive constant, that we also call $C_T$, such that
 \begin{equation}\nonumber
  \frac{1}{N}\sum_{i=1}^N\Ee\big[|M^{N}_s-M^{i,N}_s|\big]\leq \frac{C_T}{N^{1/2}}
  +\frac{C_T}{N}\sum_{i=1}^N\int_0^t\Ee\big[\big(|X^{i,N}_{s'}-Y^{i,N}_{s'}|+|M^{N}_{s'}-M^{i,N}_{s'}|\big)\big]
\,ds,
  \end{equation}
getting the

\begin{proof}[Proof of Theorem~\ref{th:chaoschaos} - (compactly supported case)] Gathering~\eqref{eq:ModulusXi} and~\eqref{eq:ModulusMn} and using the Gronwall's lemma, we get that
\begin{equation}\nonumber
 \frac{1}{N}\sum_{i=1}^N\Ee\big[\big(|X^{i,N}_s-Y^{i,N}_s|+|M^{N}_s-M^{i,N}_s|\big)\big]\,\,\leq\,\, \frac{C_T}{N^{1/2}}.
\end{equation}
To finish, we apply Fournier-Guillin~\cite[Theorem 1]{Fournier:2013kq} with $d=2$, $p=1$, $q=2+\epsilon$, to find that there exists a positive constant $C$ independent of $N$ such that
\begin{equation}\nonumber
 \Ee\big[\WW_1\big(\eta_N(t),\loi(Y^{1,N}_t,M^{1,N}_t)\big)\big]\,\leq\,C\,\Ee\big[(Y^{1,N}_t,M^{1,N}_t)^{2+\epsilon}\big]^{1/2+\epsilon}\frac{\log(1+N)}{N^{1/2}},
\end{equation}
but, since the initial laws $(g_0,m_0)$ are compactly supported, all polynomial moments of the solution are upper bounded by a constant independent of $N$. Using triangular inequality we get that
 \begin{eqnarray*}
  \Ee\big[\WW_1\big(\mu_N(t),\loi(Y^{1,N}_t,M^{1,N}_t)\big)\big]
  &\leq&\Ee\big[\WW_1\big(\mu_N(t),\eta_N(t))\big)\big]+\Ee\big[\WW_1\big(\eta_N(t),\loi(Y^{1,N}_t,M^{1,N}_t)\big)\big]
  \\
  &\leq& C_T\frac{\log(1+N)}{N^{1/2}},
 \end{eqnarray*}
 since $S$ has only one element we conclude that locally in time
 \begin{equation}\nonumber
  \mu_N\xrightarrow{N\rightarrow\infty} \loi((Y^{1,N},M^{1,N})),
 \end{equation}
 as $\log(1+N)/\sqrt{N}$.
\end{proof}

Let us now explain why we cannot conclude using the same technique in the fast decay case, i.e., to mimic the path-wise uniqueness proof. Recalling It\^o's formula~\eqref{eq:Ito-gen}, we get
\begin{equation}\nonumber
   \Ee\big[|X_t^{i,N}-Y_t^{i,N}| \big]   \leq\int_0^t \Ee\big[|a(X_{s}^{i,N},M_{s}^N)- a(Y_{s}^{i,N},M_{s}^{i,N})|\times |X_{s}^{i,N}+ Y_{s}^{i,N}|\big]\,ds.
 \end{equation}
 Moreover, thanks to the fast decay at infinite it suffices to define the event
$$
 \EE_{i,N}=\Big\{\sup_{0\leq t\leq T}|X_t^{i,N}|<R,\quad\sup_{0\leq t\leq T}|Y_t^{i,N}|<R\Big\},
$$
to get that
\begin{multline}\nonumber
   \Ee\big[|X_t^{i,N}-Y_t^{i,N}| \big]\leq 2\,C_0\,Ra(R,R)\int_0^t \Ee\big[|X_{s}^{i,N}-Y_{s}^{i,N}|\big]\,ds+2\,C_0R\int_0^t \Ee\big[|M_{s}^N-M_{s}^{i,N}|\big]\,ds
   \\
 +C_T\Pp(\EE_{i,N}^c)^{1/2}.
 \end{multline}
Finally, multiplying by $N^{-1}$ and adding on $i$, implies that
\begin{multline}\label{eq:auxfast}
   \frac{1}{N}\sum_{i=1}^N\Ee\big[|X_t^{i,N}-Y_t^{i,N}| \big]\leq 2\,C_0\,Ra(R,R)\int_0^t \frac{1}{N}\sum_{i=1}^N\Ee\big[|X_{s}^{i,N}-Y_{s}^{i,N}|\big]\,ds
   \\
+2\,C_0R\int_0^t  \frac{1}{N}\sum_{i=1}^N\Ee\big[|M_{s}^N-M_{s}^{i,N}|\big]\,ds +C_T \max_{1\leq i\leq N}\Pp(\EE_{i,N}^c)^{1/2}.
 \end{multline}

 On the other hand, the difference between $M^N_t$ and $M^{i,N}_t$, is controlled by noticing that
 \begin{multline}\nonumber
 M_t^{i,N}-M^N_t = -\alpha\int_0^t (M_s^{i,N}-M^N_s)\,ds 
 \\
 +\alpha\eps\int_0^t\int_0^s\Ee[a(Y_{s-w}^{i,N},M_{s-w}^{i,N})]\,b(dw)ds-\frac{\eps\alpha}{N}\sum_{j=1}^N \int_{0}^{t-\tau_j}a(Y^{j,N}_{s},M_{s}^{j,N})\,ds
 \\
 +\frac{\eps\alpha}{N}\sum_{j=1}^N \int_{0}^{t-\tau_j}\Big[a(Y^{j,N}_{s},M_{s}^{j,N})-a(X^{j,N}_{s},M_{s}^N)\Big]\,ds
 \\
-\frac{\eps\alpha}{N}\sum_{j=1}^N \int_{0}^{t-\tau_j}\int_{0}^\infty  \indiq_{\{u\leq a(X^{j,N}_{s},M_{s}^N)\}}\,\big[\NN^j(du,ds)-\,du\,ds\big].
 \end{multline}
The first term on the righthand side has a nice structure, and the third one is controlled by
\begin{multline}\nonumber
  \int_0^t\Ee\big[\big|a(Y_{s}^{j,N},M_{s}^{j,N})-a(X_{s}^{j,N},M_{s}^{N})\big|\big]\,ds
   \\
   \leq C_T\Pp(\EE_{j,N}^c)^{1/2}+2C_0\,Ra(R,R)\int_0^t\Ee\big[\big|X_{s}^{j,N}-Y_{s}^{j,N}\big|+\big|M_{s}^{j,N}-M_{s}^{N}\big|\big]\,ds,
 \end{multline}
which is equivalent to~\eqref{eq:auxfast}

 The other two quantities are a little bit more delicate to handle, but using exchangeability and recalling the definition of $\TTT_{1,4}(s)$ (see Proposition~\ref{prop:meanfield-Sol}), gives
 $$
  \Ee\left[\frac{\eps\alpha}N\sum_{j=1}^N\left(\int_0^t\int_0^s\Ee[a(Y_{s-w}^{j,N},M_{s-w}^{j,N})]\,b(dw)ds
- \int_{0}^{t-\tau_j}a(Y^{j,N}_{s},M_{s}^{j,N})\,ds\right)\right]=\frac{\eps\alpha}{N^{1/2}}\Ee\big[\TTT_{1,4}(s)^2\big]^{1/2},
 $$
and for the last contribution, we simply recall that the Poisson processes are independent, to get
 \begin{equation}\nonumber
  \frac{\eps\alpha}{N}\,\Ee\left[\sum_{j=1}^N\int_{0}^{t-\tau_j}\int_{0}^\infty  \indiq_{\{u\leq a(X^{j,N}_{s},M_{s}^N)\}}\,\big[\NN^j(du,ds)-\,du\,ds\big]\right]\,\leq\, \frac{\eps\alpha}{N^{1/2}}\,\int_{0}^{t}\Ee\Big[a(X^{1,N}_{s},M_{s}^N)^2\Big]^{1/2}\,ds.
 \end{equation}
 
 We gather all the previous inequalities to find
 \begin{multline}\nonumber
  \frac{1}{N}\sum_{i=1}^N\Ee\big[|X^{i,N}_s-Y^{i,N}_s|+|M^{i,N}_s-M^{N}_s|\big]
  \\
  \leq C_T \max_{1\leq i\leq N}\Pp(\EE_{i,N}^c)^{1/2}+\frac{C_T}{N^{1/2}}+C\,Ra(R,R)\int_0^t\frac{1}{N}\sum_{i=1}^N\Ee\big[|X^{i,N}_{s'}-Y^{i,N}_{s'}|+|M^{i,N}_{s'}-M^{N}_{s'}|\big]\,ds,
 \end{multline}
 but, even if we use Gronwall's lemma and condition~\eqref{eq:polyn}, we still can pass to the limit $R\rightarrow\infty$ because of the presence of the extra term $\frac{C_T}{N^{1/2}}$ on the righthand side.

\section{Further asymptotic analysis}

So far, we have seen that if initial condition are strongly bounded (in the sense of a.s. lie inside a compact) then we get a nice rate of convergence, which is somehow sharp, on the $L^1$ norm.

We explore now, under what circumstances these results still hold true in a larger space. To that aim, we use the a priori bounds and the integrability of the intensity rate. Let us start by a result that tell us that the intensity function remains with high probability bounded.

\begin{prop}\label{prop:ControlRate}
Under the assumptions~\eqref{eq:syspar-H1} and~\eqref{eq:syspar-H3}, for any positive horizon of time $T$ fixed, there is some positive constant $C_{mf}$ such that
\begin{equation}\label{eq:Pgoing0}
 \Pp\Big(\sup_{0\leq t\leq T}\max\big(\langle\mu_N(t),a^2\rangle,\langle\eta_N(t),a^2\rangle\big)\geq\, C_{mf}\Big)\xrightarrow{N\rightarrow+\infty}{}0,
\end{equation}
as $1/N$.
\end{prop}

\begin{proof}
 Let us recall that from It\^o's formula~\eqref{eq:Ito-gen} we get, for a generic function $f(\cdot,\cdot)$, that the compensated local martingale writes
 \begin{multline}\nonumber
 \MM^{i,N}_{f}(t)\,=\, \int_0^t\Big[f\big(0,M^N_{s_-+\tau_i}+\frac{\alpha\eps}{N}\indiq_{\{s\leq t-\tau_i\}}\big)-f(X^{i,N}_{s-},M^N_{s-})\Big)\Big]
 \\
 \times\Big[\int_0^\infty\indiq_{\{u\leq a(X^{i,N}_{s-},M^N_{s-})\}}\NN^{i}(du,ds)-a(X^{i,N}_s,M^N_s)\,ds\Big]
 \\
 +\sum_{j\neq i}\int_0^{t}\Big[f\big(X^{i,N}_{s_-+\tau_j},M^N_{s_-+\tau_j}+\frac{\alpha\eps}{N}\big)-f(X^{i,N}_{s_-+\tau_j},M^N_{s_-+\tau_j})\Big]\indiq_{\{s\leq t-\tau_j\}}
 \\\times\Big[\int_0^\infty\indiq_{\{u\leq a(X^{j,N}_{s-},M^N_{s-})\}}\NN^j(du,ds)-a(X^{j,N}_{s},M^N_{s})\,ds\Big].
\end{multline}
Using that the local martingale $\MM^{i,N}_{a^2}(t)$ is locally integrable and that $a(0,\cdot)=0$, the associated brackets processes, in the case $f=a^2$, are given by
\begin{multline}\nonumber
 \langle\MM^{i,N}_{a^2}\rangle(t)\,=\,
 \int_0^t a(X^{i,N}_s,M^N_s)^5\,ds
 \\
 +\sum_{j\neq i}\int_0^{t}\Big[a\big(X^{i,N}_{s_-+\tau_j},M^N_{s_-+\tau_j}+\frac{\alpha\eps}{N}\big)^2-a(X^{i,N}_{s_-+\tau_j},M^N_{s_-+\tau_j})^2\Big]^2\indiq_{\{s\leq t-\tau_j\}}a(X^{j,N}_{s},M^N_{s})\,ds,
 \end{multline}
and, for $i\neq j$,
 \begin{multline}\nonumber
 \langle\MM^{i,N}_{a^2},\MM^{j,N}_{a^2}\rangle(t)\,=\,
 \\
 - \int_0^ta(X^{i,N}_s,M^N_s)^2\Big[a\big(X^{j,N}_{s_-+\tau_i},M^N_{s_-+\tau_i}+\frac{\alpha\eps}{N}\big)^2-a(X^{j,N}_{s_-+\tau_i},M^N_{s_-+\tau_i})^2\Big]\indiq_{\{s\leq t-\tau_i\}}a(X^{i,N}_s,M^N_s)\,ds
  \\
 - \int_0^ta(X^{j,N}_s,M^N_s)^2\Big[a\big(X^{i,N}_{s_-+\tau_j},M^N_{s_-+\tau_j}+\frac{\alpha\eps}{N}\big)^2-a(X^{i,N}_{s_-+\tau_j},M^N_{s_-+\tau_j})^2\Big]\indiq_{\{s\leq t-\tau_j\}}a(X^{j,N}_s,M^N_s)\,ds
 \\
 +\sum_{k\neq i, j}\int_0^{t}\Big[a\big(X^{i,N}_{s_-+\tau_k},M^N_{s_-+\tau_k}+\frac{\alpha\eps}{N}\big)^2-a(X^{i,N}_{s_-+\tau_k},M^N_{s_-+\tau_k})^2\Big]
 \\
 \times\Big[a\big(X^{j,N}_{s_-+\tau_k},M^N_{s_-+\tau_k}+\frac{\alpha\eps}{N}\big)^2-a(X^{j,N}_{s_-+\tau_k},M^N_{s_-+\tau_k})^2\Big]\indiq_{\{s\leq t-\tau_k\}}a(X^{k,N}_{s-},M^N_{s-})\,ds.
\end{multline}

Now, define $\MM^N_{a^2}(t)=N^{-1}\sum_{i=1}^N\MM^{i,N}_{a^2}(t)$,  it follows that
\begin{equation}\nonumber
 \Ee\big[\big(\sup_{0\leq s\leq t}\MM^N_{a^2}(t)\big)^2\big]=\Ee\big[\langle\MM_{a^2}^N\rangle(t)\big]\,=\,\frac{1}{N^2}\Ee\Bigg[\sum_{i=1}^N \langle\MM_{a^2}^{i,N}\rangle(t)+2\sum_{ i<j} \langle\MM^{i,N}_{a^2},\MM^{j,N}_{a^2}\rangle(t)\Bigg]\leq \frac{C_T}{N},
\end{equation}
where the last inequality is obtained using the same arguments of Proposition~\ref{lem:apriori-PS2}. In particular, we get that
\begin{equation}\nonumber
 \Pp\big[\sup_{0\leq t\leq T} \MM^N_{a^2}(t)\geq 1\big]\leq \Ee\big[\big(\sup_{0\leq t\leq T}\MM^N_{a^2}(t))^2\big]\leq  \frac{C_T}{N}.
\end{equation}

We come back to~\eqref{eq:Ito-gen} in the case $f=a^2$ to get
\begin{multline}\nonumber
 \langle \mu_N(t),a^2\rangle\,\leq\, \langle \mu_N(0),a^2\rangle
 +2C_0\,\int_0^t\langle\mu_N(s),a^2\rangle\,ds+\MM^N_{a^2}(t)
\\
 +2\,C_0\,\alpha\eps\int_0^{t}\sup_{0\leq s'\leq s}\langle\mu_N(s'),a\rangle\langle\mu_N(s),a\rangle\,ds,
\end{multline}
therefore, conditioning on the event $\sup_{0\leq t\leq T} \MM^N_{a^2}(t)\leq 1$ we get
\begin{equation}\nonumber
\sup_{0\leq s\leq t}\langle \mu_N(s),a^2\rangle\,\leq\,1+C_0\alpha\eps+\langle \mu_N(0),a^2\rangle
 +2\,C_0(1+2\,\alpha\eps)\,\int_0^t\sup_{0\leq s'\leq s}\langle\mu_N(s'),a^2\rangle\,ds.
\end{equation}
Finally, we recall that
\begin{equation}\nonumber
 \langle \mu_N(0),a^2\rangle\,=\,\frac1N\sum_{i=1}^N a(X_0^{i,N},M^N_0)^2,
\end{equation}
and by the law of large numbers we realise that there is a constant $C_T$ such that
\begin{equation}\nonumber
\Pp\big[\langle\mu_N(0),a^2\rangle>1+\Ee[a^2(X^{1,1}_0,M^1_0)\rangle]\big]<\frac{C_T}{N},
\end{equation}
and we deduce that indeed $\langle\mu_N,a^2\rangle$ remains with high probability bounded.

\bigskip

Finally, let us recall that
\begin{multline}\nonumber
 a(Y^{i,N}_t,M^{i,N}_t)^2\,\leq\,a(X^{i,N}_0,M^{N}_0)^2+C_0\int_0^ta(Y^{i,N}_s,M^{i,N}_s)^2\,ds
  \\
 +C_0\int_0^{t}a(Y^{i,N}_s,M^{i,N}_s)\int_0^s\Ee\big[a(Y_{s-w}^{i,N},M_{s-w}^{i,N})\big]\,b(dw)\,ds\,+\bar\MM^{i,N}_{a^2}(t),
\end{multline}
and $\int_0^t\Ee\big[a(Y_{s-w}^{i,N},M_{s-w}^{i,N})\big]\,b(dw)\,ds<C_T$, with an upper bound independent of $N$. Moreover, 
\begin{equation}\nonumber
\langle\bar\MM_{a^2}^{i,N}\rangle(t)\,=\, \int_0^t a(Y^{i,N}_s,M^{i,N}_s)^5\,ds,
\end{equation}
which in expectation is bounded thanks to the fast decay at infinite of the initial conditions. Then, from the independency of the Poisson processes, we also get
\begin{equation}\nonumber
 \Pp\big[\sup_{0\leq t\leq T} \bar\MM^N_{a^2}(t)\geq 1\big]\leq \Ee\big[\big(\sup_{0\leq t\leq T}\bar\MM^N_{a^2}(t)\big)^2\big]=\frac{1}{N^2}\Ee\Bigg[\sum_{i=1}^N \langle\bar\MM_{a^2}^{i,N}\rangle(t)\Bigg]\leq \frac{C_T}{N}.
\end{equation}
\end{proof}

Finally, we have all conditions to quantify the convergence of the empirical measures in the space of functions with fast decay at infinity.

\begin{proof}[Proof of Theorem~\ref{th:chaoschaos} - (fast decay at infinite case)] 
%

Starting once again with the It\^o's formula~\eqref{eq:Ito-gen}, we have that
\begin{equation}\nonumber
   \Ee\big[|X_t^{i,N}-Y_t^{i,N}| \big]\leq \int_0^t \Ee\Big[|a(X_{s}^{i,N},M_{s}^N)- a(Y_{s}^{i,N},M_{s}^{i,N})|\times |X_{s}^{i,N}+ Y_{s}^{i,N}|\Big]\,ds.
 \end{equation}
  Moreover, defining the event
$$
 \EE_N=\Big\{\sup_{0\leq t\leq T}\langle\mu_N(t),a^2\rangle<C_{mf},\quad\sup_{0\leq t\leq T}\langle\eta_N(t),a^2\rangle<C_{mf}\Big\},
$$
we get that
\begin{multline}\nonumber
   \Ee\big[|X_t^{i,N}-Y_t^{i,N}| \big] \leq C_T\Pp(\EE_N^c)^{1/2}
   \\
   +\int_0^t\Ee\Big[|a(X_{s}^{i,N},M_{s}^N)-a(Y_{s}^{i,N},M_{s}^{i,N})|\times\big(X_{s}^{i,N}+Y_{s}^{i,N}\big)\indiq_{\EE_{N}}\Big]\,ds
 \end{multline}
From hypothesis~\eqref{eq:polyn} and Jensen's inequality, we have
\begin{equation}\nonumber
\frac1N\sum_{i=1}^N |X^{i,N}_t|^{2(1+\rho)}\leq \frac{c_\rho^{1-\rho}}{N}\sum_{i=1}^Na(X^{i,N}_t,M^N_t)^{2(1-\rho)}
\leq  \frac{c_\rho^{1-\rho}}{N^\rho}\,\big[\langle\mu_N(t),a^2\rangle\big]^{1-\rho},
 \end{equation}
therefore, multiplying by $N^{-1}$ and adding on $i$, we conclude
\begin{multline}\nonumber
   \frac{1}{N}\sum_{i=1}^N\Ee\Big[|X_t^{i,N}-Y_t^{i,N}| \Big]
   \leq C_T\,\Pp(\EE_N^c)^{1/2}
   \\
+\frac{4}{N}\sum_{i=1}^N\int_0^t\Ee\Big[\Big(a(X_{s}^{i,N},M_{s}^N)^{2(1-\rho)}+a(Y_{s}^{i,N},M_{s}^{i,N})^{2(1-\rho)}+|X_{s}^{i,N}|^{2(1+\rho)}+|Y_{s}^{i,N}|^{2(1+\rho)}\Big)\indiq_{\EE_N}\Big]\,ds,
 \end{multline}
 therefore there is a positive constant $C_T$, not depending on $N$, such that
 \begin{equation}\nonumber
   \frac{1}{N}\sum_{i=1}^N\Ee\Big[|X_t^{i,N}-Y_t^{i,N}| \Big] \leq C_T\,\Pp(\EE_N^c)^{1/2}+\frac{C_T}{N^\rho}.
 \end{equation}

 On the other hand, the difference between $M^N_t$ and $M^{i,N}_t$, is controlled by noticing that
 \begin{multline}\nonumber
 M_t^{i,N}-M^N_t = -\alpha\int_0^t (M_s^{i,N}-M^N_s)\,ds 
 \\
 +\alpha\eps\int_0^t\int_0^s\Ee[a(Y_{s-w}^{i,N},M_{s-w}^{i,N})]\,b(dw)ds-\frac{\eps\alpha}{N}\sum_{j=1}^N \int_{0}^{t-\tau_j}a(Y^{j,N}_{s},M_{s}^{j,N})\,ds
 \\
 +\frac{\eps\alpha}{N}\sum_{j=1}^N \int_{0}^{t-\tau_j}\Big[a(Y^{j,N}_{s},M_{s}^{j,N})-a(X^{j,N}_{s},M_{s}^N)\Big]\,ds
 \\
-\frac{\eps\alpha}{N}\sum_{j=1}^N \int_{0}^{t-\tau_j}\int_{0}^\infty  \indiq_{\{u\leq a(X^{j,N}_{s},M_{s}^N)\}}\,\big[\NN^j(du,ds)-\,du\,ds\big].
 \end{multline}
The first term on the righthand side has a nice structure, and the third one is controlled by
\begin{multline}\nonumber
  \int_0^t\Ee\big[\big|a(Y_{s}^{j,N},M_{s}^{j,N})-a(X_{s}^{j,N},M_{s}^{N})\big|\big]\,ds
   \\
   \leq C_T\Pp(\EE_N^c)^{1/2}+C_0\int_0^t\Ee\big[\big|M_{s}^{j,N}-M_{s}^{N}\big|\big]\,ds
   \\
   +C_0\int_0^t\Ee\Big[\big(a(X_{s}^{j,N},M_{s}^N)+a(Y_{s}^{j,N},M_{s}^{j,N})\big)\times\big(X_{s}^{j,N}+Y_{s}^{j,N}\big)\indiq_{\EE_{N}}\Big]\,ds.
 \end{multline}

 The other two quantities are a little bit more delicate to handle, but using exchangeability and recalling the definition of $\TTT_{1,4}(s)$ (see Proposition~\ref{prop:meanfield-Sol}), gives
 $$
  \Ee\left[\Big|\frac{\eps\alpha}N\sum_{j=1}^N\left(\int_0^t\int_0^s\Ee[a(Y_{s-w}^{j,N},M_{s-w}^{j,N})]\,b(dw)ds
- \int_{0}^{t-\tau_j}a(Y^{j,N}_{s},M_{s}^{j,N})\,ds\right)\Big|\right]=\frac{\eps\alpha}{N^{1/2}}\Ee\big[\TTT_{1,4}(s)^2\big]^{1/2},
 $$
and for the last contribution, we simply recall that the Poisson processes are independent, to get
 \begin{equation}\nonumber
  \frac{\eps\alpha}{N}\,\Ee\left[\Big|\sum_{j=1}^N\int_{0}^{t-\tau_j}\int_{0}^\infty  \indiq_{\{u\leq a(X^{j,N}_{s},M_{s}^N)\}}\,\big(\NN^j(du,ds)-\,du\,ds\big)\Big|\right]\,\leq\, \frac{\eps\alpha}{N^{1/2}}\,\int_{0}^{t}\Ee\Big[a(X^{1,N}_{s},M_{s}^N)^2\Big]^{1/2}\,ds.
 \end{equation}
 
 Gathering all the previous inequalities leads to
 \begin{equation}\nonumber
  \Ee\big[|X^{i,N}_s-Y^{i,N}_s|+|M^{i,N}_s-M^{N}_s|\big]
  \leq \frac{C_T}{N^{1/2}}+\frac{C}{N^\rho}+C_T\Pp(\EE_N^c)^{1/2}+C_T\int_0^t\Ee\big[|M^{i,N}_{s'}-M^{N}_{s'}|\big)\big]\,ds,
 \end{equation}
 from where we finally get that
 \begin{equation}\nonumber
  \Ee\big[\big(|X^{i,N}_s-Y^{i,N}_s|^2+|M^{i,N}_s-M^{N}_s|^2\big)\big]
  \\
  \leq C_Te^{TC_T}\max\Big(\frac1{N^{1/2}},\frac1{N^\rho}\Big),
 \end{equation}
 we notice that the best convergence rate is in the case of $\rho=1/2$. 
\end{proof}

\bigskip
\bigskip
\appendix

\section{General Theorems for Stochastic Processes}
\label{sec:AppJacob}
\setcounter{equation}{0}
\setcounter{theo}{0}

\subsection*{Remarks on the proof of Lemma~\ref{lem:conditions-T}} Let us recall that the last thing we got in the main text was that for any $\varphi\in C_b^2(\R_+^2)$, the process
 \begin{multline}\label{eq:MartingalePb}
  \varphi(Y_t,M_t)-\varphi(Y_0,M_0)-\int_0^t\partial_m\varphi(Y_s,M_s)\Big[-\alpha M_s+\alpha\,\eps\int_0^\tau\Ee[a(Y_{s-s'},M_{s-s'})]\,b(ds')\Big]ds
  \\
  -\int_0^t\partial_y\varphi(Y_s,M_s)\,ds-\int_0^t a(Y_s,M_s)\big(\varphi(0,M_s)-\varphi(Y_s,M_s)\big)ds,
 \end{multline}
 is a local martingale.

Let us recall now the Jacob-Shiryaev~\cite[Theorem II.2.42 page 86]{MR1943877}
\begin{theo}
There is equivalence between:
\begin{itemize}
 \item $(Y_t,M_t)$ is a semimartingale, and it admits the characteristics $(B,0,\nu)$; i.e., $(Y_t,M_t)$ writes
 $$
  (Y_t,M_t)\,=\, (Y_0,M_0)+\MM^c+B,
 $$
 where $\MM^c$ is the continuous local martingale of the canonical decomposition, $B$ is predictable and $\nu$ is a predictable random measure on $\R_+\times\R_+^2$, namely the compensator of the random measure associated to the jumps of $X$.
 \item For each bounded function $\varphi\in C^2(\R_+^2)$, the process
 \begin{multline}\nonumber
  \varphi(Y_t,M_t)-\varphi(Y_0,M_0) - \int_0^t\partial_y\varphi(Y_{s-},M_{s-})\, dB^y_s-\int_0^t\partial_m\varphi(Y_{s-},M_{s-})\, dB^m_s
  \\
  -\int_0^t\left\{\varphi(Y_{t-}+y,M_{t-}+x)-\varphi(Y_{t-},M_{t-})-y\,\partial_y\varphi(Y_{t-},M_{t-})-m\,\partial_m\varphi(Y_{t-},M_{t-})\right\}\nu(ds,dy,dm)
 \end{multline}
 is a local martingale.
\end{itemize}
\end{theo}

\noindent Then, in our case of study, by choosing the characteristics
\begin{equation}\nonumber
 B^y_t\,=\,\int_0^t\Big[1+Y_s\,a(Y_{s-},M_{s-})\Big]\,ds,\quad
 B^m_t\,=\,\int_0^t\Big[-\alpha M_s+\alpha\,\eps\int_0^s\Ee[a(Y_{s-s'},M_{s-s'})]\,b(ds')\Big]\,ds,
\end{equation}
and by
\begin{equation}\nonumber
 \nu(ds,dy,dm)\,=\,a(Y_{s-},M_{s-})\,ds\,\delta_{-Y_{s-}}(dy)\delta_{0}(dm),
\end{equation}
we get that $(Y_t,M_t)$ is a semi martingale.
\bigskip

As for the second important Jacob-Shiryaev~\cite[Theorem III.2.26 page 157]{MR1943877} cited in the main text, let us now rewrite their general result to our study case. Consider the stochastic differential equation
\begin{equation}\label{eq:Jacob-gen}
\begin{cases}
 (Y_0,M_0)\,=\,(\xi_y,\xi_m)
 \\
 d(Y_t,M_t)\,=\,\beta(t,Y_t,M_t)\,dt+\delta(t,Y_{t-},M_{t-},z)\big(\NN(du,dt)-q(du,dt)\big),
\end{cases}
\end{equation}
where $\NN$ is a standard Poisson process with intensity measure $q(du,dt)=du\,dt$.

\begin{theo}
 Let $\eta$ be a suitable initial condition (i.e., a probability on $\R_+^2)$, and $\beta,\delta$ be
 \begin{equation}\nonumber
 \begin{cases}
  \beta\,=\,(\beta^1,\beta^2),\text{ a Borel function: }\R_+\times \R_+^2\rightarrow\R_+^2,
  \\
  \delta\,=\,(\delta^1,\delta^2),\text{ a Borel function: }\R_+\times \R_+\times \R_+^2\rightarrow\R_+^2.
 \end{cases}
 \end{equation}
 
 The set of all solutions to~\eqref{eq:Jacob-gen} with initial condition $\eta$ is the set of all solutions to a martingale problem on the canonical space where the characteristics $(B,0,\nu)$ are given by
 \begin{equation}\nonumber
  B^i_t(w)\,=\,\int_0^t \beta^i(s,Y_s(w),M_s(w))\,ds,\qquad \nu(w,dt\times dy\,dm)\,=\,dt\times K_t(Y_t(w),M_t(w),dy,dm),
 \end{equation}
 with
 \begin{equation}\nonumber
  K_t(y,m,A)\,=\,\int_0^\infty\indiq_{\{A\setminus\{0\}\}}(\delta(t,u,y,m))\,du.
 \end{equation}
\end{theo}

We notice that $(Y_t,M_t)$ indeed solves the Martingale problem given by~\eqref{eq:MartingalePb}, therefore it is a solution to the equation~\eqref{eq:Jacob-gen} for some standard Poisson process, and therefore Lemma~\ref{lem:conditions-T} is proved.

\bibliographystyle{acm}
\bibliography{./PoissonLimit}

\bigskip\bigskip
 \signcq  

\end{document}